\documentclass[11pt]{amsart}

\usepackage[utf8]{inputenc} 
\usepackage{mathtools}
\usepackage{amssymb}
\usepackage{amsmath}
\usepackage{todonotes}
\usepackage{amsthm}
\usepackage{listings}
\usepackage{physics}
\usepackage{calrsfs}
\usepackage{enumerate}
\usepackage{tikz-cd}
\usepackage{mathrsfs}
\usepackage{bm}
\usepackage{esvect}
\usepackage{enumitem}
\usepackage{hyperref}
\usepackage[final]{pdfpages}
\usepackage{verbatim}

\usepackage[backend=biber, style=numeric, sorting=nyt]{biblatex}
 \usepackage[parfill]{parskip}
 
\newtheorem{theorem}{Theorem}[section]
\newtheorem{lemma}[theorem]{Lemma}
\newtheorem{proposition}[theorem]{Proposition}
\newtheorem{corollary}[theorem]{Corollary}

\newtheorem{conjecture}[theorem]{Conjecture}

\newtheorem{question}[theorem]{Question}
\theoremstyle{remark}
\newtheorem{remark}[theorem]{Remark}

\numberwithin{equation}{section}

\newcommand{\R}{\mathbb{R}}
\newcommand{\C}{\mathbb{C}}
\newcommand{\Z}{\mathbb{Z}}
\newcommand{\N}{\mathbb{N}}
\renewcommand{\H}{\mathbb{H}}
\newcommand{\Q}{\mathbb{Q}}
\newcommand{\sym}{\mathrm{sym}}

\def\Cl{\text{\rm Cl}}

\def\sym{\mathrm{sym}}

\def\PSL{\mathrm{PSL}}

\def\CC{\mathcal{C}}
\def\A{\mathbb{A}}

\DeclareMathOperator{\vol}{vol}

\DeclareFontFamily{U}{wncy}{}
    \DeclareFontShape{U}{wncy}{m}{n}{<->wncyr10}{}
    \DeclareSymbolFont{mcy}{U}{wncy}{m}{n}
    \DeclareMathSymbol{\Sh}{\mathord}{mcy}{"58} 
\def\modulo{\text{ \rm mod }}
\def\vol{\text{\rm vol}}

\newcommand{\lr}[1]{\left (   {#1} \right )}

\newcommand{\inprod}[2]{\left \langle  {#1} , {#2} \right \rangle}

\newcommand{\mods}[1]{\,(\mathrm{mod}\,{#1})}

\newcommand{\Li}{\mathrm{Li}}

\usepackage{geometry} 
\usepackage[stretch=10]{microtype}
\usepackage{xcolor}
\definecolor{dark-red}{rgb}{0.4,0.15,0.15}
\definecolor{dark-blue}{rgb}{0.15,0.15,0.4}
\definecolor{medium-blue}{rgb}{0,0,0.5}
\hypersetup{
	pdftitle={Non-vanishing geodesic periods},
	pdfauthor={Petru Constantinescu},
	pdfnewwindow=true,
	colorlinks, linkcolor={dark-red},
	citecolor={dark-blue}, urlcolor={medium-blue}
}

\newcommand{\CLK}{\mathrm{Cl}_K^+}

\newcommand{\ed}{\epsilon_D} 

\newcommand{\DNf}{\mathcal{D}^{N}_{\mathrm{fund}}}
\newcommand{\Df}{\mathcal{D}_{\mathrm{fund}}}

\title{Non-vanishing of geodesic periods of automorphic forms}

\begin{document}
\author{Petru Constantinescu and Asbj\o rn Christian Nordentoft}

\address{EPFL SB MATH TAN,
Station 8,
1015 Lausanne, Switzerland}

\email{\href{mailto:petru.constantinescu@epfl.ch}{petru.constantinescu@epfl.ch}}

\address{Universit\'{e} Paris-Saclay, Laboratoire de Math\'{e}matiques d'Orsay, 307, rue Michel Magat, 91405 Orsay Cedex, France}

\email{\href{mailto:acnordentoft@outlook.com}{acnordentoft@outlook.com}}

\date{\today}


\begin{abstract}
We prove that one hundred percent of the closed geodesic periods of a Hecke--Maa{\ss} cusp form for the modular group are non-vanishing when ordered by length. We present applications to the non-vanishing of central values of Rankin--Selberg $L$-functions. Similar results for holomorphic forms for general Fuchsian groups of finite covolume with a cusp are also obtained, as well as results towards normal distribution. Our new key ingredient is to relate the distributions of closed geodesic periods and vertical line integrals via graph theory.
\end{abstract}
\maketitle
\section{Introduction}

The study of closed geodesics on the modular surface $\PSL_2(\Z) \backslash \H$ is a rich and important subject, at the confluence of arithmetic, geometry, and dynamics \cite{MichelVenk06}, \cite{EinLindMichVenk12}. In particular, the closed geodesics encode deep arithmetic information via Waldspurger's formula. More precisely, let $K=\Q(\sqrt{D})$ be a real quadratic field of discriminant $D>0$. For each narrow ideal class $A \in \CLK$, one can associate an oriented closed geodesic $\mathcal{C}_A \subset \PSL_2(\Z) \backslash \H$ of length $2\log \epsilon_K$, where $\epsilon_K=u+\sqrt{D}v$ with $u^2-Dv^2=1$ denotes the fundamental positive unit of $K$ (see e.g. \cite{sarnak82}). We know from the class number formula and a result of Siegel that $|\CLK| \log \epsilon_K =D ^{1/2+o(1)}$. Let $f$ be a Hecke--Maa{\ss} cusp form of weight $0$, level $1$ and sign $\epsilon(f)\in \{\pm 1\}$. Let $\chi :\CLK\rightarrow \C^\times$ be a narrow class group character such that $\epsilon(f)=\chi(J)$, where $J=(\sqrt{D})\in \CLK$ denotes the different of $K$. Then we have by Waldspurger's formula, due to Popa \cite[Thm.\ 6.3.1]{popa06} in its explicit form:
\begin{equation}\label{eq:popa} L(f\otimes\theta_\chi,1/2)=\frac{c^+_f}{D^{1/2}}\left| \sum_{A\in \CLK}\chi(A) \int_{\CC_A} f(z) \frac{|dz|}{\Im z} \right|^2,\end{equation}
where $\theta_{\chi}$ is the theta series associated to $\chi$, $L(f\otimes \theta_\chi,s)$ is the Rankin--Selberg $L$-function of $f$ and $\theta_\chi$, and $c^+_f>0$ is some constant depending only on $f$. A similar formula holds in the case $\epsilon(f)=-\chi(J)$, see \cite[Lemma 6.14]{HumNor20}.

The present paper is concerned with the study of the arithmetic statistics of the \emph{geodesic periods} $\int_{\CC_A} f(z) \frac{|dz|}{\Im z}$ and their generalization to general Fuchsian groups and general automorphic forms. More precisely, our work was motivated by the following question\footnote{Note that the question as stated in \emph{loc.\ cit.} needs slight modification for odd Hecke--Maa{\ss} forms.} posed by Michel \cite[p. 1377]{Oberwolfach20}:

\begin{question}[Michel]\label{q:michel}
Fix $\delta>0$. Let $f$  be a Hecke--Maa{\ss} cusp form of weight $0$, level $1$ and sign $\epsilon(f)\in \{\pm 1\}$. Let $K$ be a real quadratic field of discriminant $D$ such that $|\Cl_K^+|\geq D^\delta$. If $\epsilon(f)=-1$, assume that $J\neq (1)\in \Cl_K^+$. For $D$ large enough, does there always exist $A\in \Cl_K^+$ such that $\int_{\mathcal{C}_A}f (z) \frac{|dz|}{\Im z}\neq 0$? 
Equivalently, does there exist a narrow class group character $\chi:\Cl_K^+\rightarrow \C^\times$ with $\chi(J)=\epsilon(f)$ such that $L(f\otimes \theta_\chi,1/2)\neq 0$?
\end{question}

Here the equivalence between the two non-vanishing questions follows by Waldspurger's formula (\ref{eq:popa}) and character orthogonality. The motivation for this question is the work of Michel--Venkatesh \cite{MichelVenkatesh07} for the imaginary quadratic analogue. In this case, the geodesic periods correspond to evaluating the Hecke--Maa{\ss} cusp form of weight $0$ at Heegner points.  Michel--Venkatesh combined subconvexity with equidistribtion of Heegner points (as proved by Duke \cite{Duke88}) to obtain non-vanishing result for the corresponding Rankin--Selberg $L$-functions. They also noticed that the corresponding argument using equidistribution of closed geodesics (also proved by Duke in \emph{loc.\ cit.}) falls short ultimately due to the infinitude of units for real quadratic fields. For more details see Section \ref{sec:MicVen}.  

\begin{remark}
  We note that for \emph{odd} forms  the condition $J\neq (1)$, as well as some ``non-smallness'' assumption is necessary in Question \ref{q:michel}: Let $f$ be an odd Hecke--Maa{\ss} cusp form of weight $0$ and level $1$. Then it holds for any ideal class $A\in \CLK$ that
  \begin{align}
      \int_{\CC_A} f(z) \frac{|dz|}{\Im z}= - \int_{\CC_A} f(-\overline{z}) \frac{|dz|}{\Im z} =-\int_{-\overline{\CC_A}} f(z) \frac{|dz|}{\Im z}.
  \end{align}
  It follows from \cite[eq.\ (2.4)]{DuImTo}   that $-\overline{\CC_A}$ equals the closed geodesic $\CC_{JA}$. Thus if the different $J$ is trivial in the narrow class group $\Cl_K^+$, then all the geodesic periods vanish for odd forms. Furthermore, by \cite[eq.\ (2.5)]{DuImTo}, the geodesics  $\CC_{JA}$ and $\CC_{A^{-1}}$ are the same but with opposite orientations, and thus the two associated geodesic periods are the same. If $\CLK$ is an elementary 2-group (which conjecturally should happen infinitely often \cite{CohenLenstra84}), we have $A^{-1}=A$ for all $A\in \CLK$. Hence, in this case, all the geodesic periods vanish for odd forms, even if $J$ is non-trivial. By genus theory, $\CLK$ is an elementary $2$-group exactly if $|\CLK|=2^{\omega(D)-1}$, where $\omega(D)$ denotes the number of prime divisors of $D$. Note that this does not contradict Question \ref{q:michel} since indeed it holds that $2^{\omega(D)-1}\ll  D^\varepsilon$ for any $\varepsilon>0$. For \emph{even} Hecke--Maa{\ss} cusp forms it is conceivable that all the geodesic periods are non-vanishing.   
\end{remark}

\subsection{Main results} In this paper we introduce a new method which answers an average version of Michel's question in a strong sense. In particular, we show that, when ordered by length, $100\%$ of the geodesic periods are non-vanishing. More precisely, let 
\begin{equation}
    \label{eq:CX}
C (X):= \{ \mathcal{C} \subset  \PSL_2(\Z) \backslash \H\text{ primitive closed geodesic}:  \ell(\mathcal{C}) \leq X\}\end{equation}  
denote the set of primitive closed geodesics on the modular curve of length $\ell(\mathcal{C}):=\int_\CC 1 \tfrac{|dz|}{\Im z}$ at most $X$.
From the prime geodesic theorem, with the best error term to date given by \cite{soyo}, we know for any $\epsilon>0$ that
\begin{equation*}
    |C(X)| = \Li(e^X) + O_\epsilon \lr{e^{X(25/36 + \epsilon)}}=\frac{e^X}{X}(1+o(1)),\quad \text{when }X\rightarrow \infty.
\end{equation*}
Our first main result is the following quantitative bound for the vanishing set.
\begin{theorem}
    \label{main thm nonvanishing}
    Let $f$ be a non-zero Hecke--Maa{\ss} cusp form of weight $0$ and level $1$. Then
    \begin{equation*}
        \left|\left \{ \mathcal{C} \in C (X): \int_{\mathcal{C}}f (z) \tfrac{|dz|}{\Im z}= 0   \right \}\right| \ll \frac{e^X}{ X^{5/4}}.
    \end{equation*}
     
\end{theorem}

\begin{remark}
    We obtain a similar result for holomorphic forms of weight $k \in 2\Z$, as well as bounds for the size of the set of closed geodesics with small (resp.\ large) geodesic period, see Theorems \ref{thm:main} and \ref{small length bound}.
\end{remark}

\subsubsection{Result for general Fuchsian groups} More generally, given a Fuchsian group $\Gamma\leq \PSL_2(\R)$ of finite covolume, we have an associated hyperbolic $2$-orbifold $\mathcal{X}_\Gamma:=\Gamma\backslash \H$. Let $f$ be an automorphic form of weight $k\in 2\Z$ for $\Gamma$, as in eq.\ (\ref{eq:automorphy}). Denote by $F:\Gamma\backslash \PSL_2(\R)\rightarrow \C$ the lift of $f$ to the unit tangent bundle of $\mathcal{X}_\Gamma$. The primitive oriented closed geodesics on $\mathcal{X}_\Gamma$ are in one-to-one correspondence with primitive hyperbolic conjugacy classes of $\Gamma$. Given an oriented closed geodesic $\CC$, we can lift it canonically to the unit tangent bundle (which we will denote by the same symbol). We equip $\CC$ with the unique $A$-invariant measure $\mu_\CC$ with volume $\ell(\CC)$, where $A$ denotes the diagonal subgroup of $\PSL_2(\R)$. We refer to Section \ref{sec:background} for more details. A key quantity of study are the geodesic periods 
\begin{equation}\mathcal{P}_f(\CC):=\int_\CC F(g) d\mu_\CC(g).\end{equation}
A first natural question is to fix a closed geodesic and ask how the sizes of the geodesic periods behave as the automorphic form $f$ is varying, see \cite{sxz}, \cite{reznikov}, \cite{BlomerBrumleyKontorovichTemplier14}, \cite{Michels22} for results in this direction. 
In this paper we are interested in obtaining strong lower bound for the set of non-vanishing geodesic periods of a fixed automorphic form. To quantify this, put for $X\geq 1$: 
\begin{equation}
    \label{eq:CgammaX}C_\Gamma(X):= \left\{ \CC\subset \mathcal{X}_\Gamma \text{ primitive closed geodesic}: \ell(\CC)\leq X \right\}.\end{equation} 
    The first general result in this direction seems to be the work of Katok \cite{Katok85}, who proved that for a holomorphic automorphic form $f$ on $\Gamma$ (including the co-compact case), there exists at least one non-vanishing geodesic period. The proof  relies on Poincar\'{e} series techniques. Zelditch used a variant of the Selberg trace formula to estimate the first moment of the geodesic periods (averaged over $C_\Gamma(X)$) \cite[Thm.\ 5A]{zel92}, which implies that if the Laplacian on $\mathcal{X}_\Gamma$ admits an eigenfunction with eigenvalue $0<\lambda<\tfrac{3}{16}$, then any Maa{\ss} cusp form for $\Gamma$ has infinitely many non-vanishing geodesic periods (and similarly results for Eisenstein series using \cite[Thm.\ 5B]{zel92}). In the case of weight $2$ holomorphic forms, Petridis--Risager \cite{PR05} used spectral methods to prove that the geodesic periods for co-compact $\Gamma$ when ordered by geodesic length are normally distributed (and in particular $100\%$ of them are non-vanishing).

Our second main result is the following general non-vanishing theorem in the presence of a cusp, more specifically we show that $100\%$ of the geodesic periods are of the expected size.  We emphasize that in the case of weight $2$ holomorphic forms, the techniques are quite different from those in \cite{PR05}, as we crucially use the presence of a cusp, see Section \ref{sec:ideaofproof} for details on our approach.  
\begin{theorem}\label{thm:main}
Let $\Gamma\leq \PSL_2(\R)$ be a Fuchsian group of finite covolume with a cusp. Let $f$ be a non-zero automorphic form for $\Gamma $ such that either 
\begin{itemize}
    \item $\Gamma=\PSL_2(\Z)$ and $f$ is a Hecke--Maa{\ss} cusp form of weight $0$, or
    \item $y^{-k/2}f$ is a holomorphic cusp form of weight $k\in 2\Z_{>0}$. 
\end{itemize}

Then $100\%$ of the geodesic periods of $f$ are non-vanishing when ordered by length. More precisely, for any function $h:\R_{\geq 0}\rightarrow \R_{\geq 0}$ such that $h(x)\rightarrow \infty$ as $x\rightarrow \infty$ it holds that 
\begin{align}
    \frac{\left|\left\{ \CC\in C_\Gamma(X): h(X)^{-1}\leq \left | \mathcal{P}_f(\CC) \right |/\sqrt{X}\leq h(X) \right\}\right|}{|C_\Gamma(X)|}\rightarrow 1,
\end{align}
as $X\rightarrow \infty$.
\end{theorem}

We obtain a similar result for Eisenstein series for $\Gamma=\PSL_2(\Z)$.

\begin{theorem}
\label{thm:Eis}
 Let $E^\ast(z):=E^\ast_{\infty,0}(z, 1/2)$ denote the completed Eisenstein series of weight $0$ for $\PSL_2(\Z)$ at the central point $s=1/2$. For any function $h:\R_{\geq 0}\rightarrow \R_{\geq 0}$ such that $h(x)\rightarrow \infty$ as $x\rightarrow \infty$, it holds that
    \begin{align}
    \frac{\left|\left\{ \CC\in C_\Gamma(X): h(X)^{-1}\leq \left |\mathcal{P}_{E^\ast}\left (\CC \right ) \right |/(\sqrt{X}(\log X)^{3/2})\leq h(X) \right\}\right|}{|C_\Gamma(X)|}\rightarrow 1,
\end{align}
as $X\rightarrow \infty$.
\end{theorem}

\subsubsection{Applications to non-vanishing of Rankin--Selberg $L$-values}Restricting Theorem \ref{thm:main} to the case of Hecke congruence groups and using a result of Raulf \cite{raulf16}, we obtain the following non-vanishing theorem for central $L$-values, which is an averaged real quadratic analogue of the results of Michel--Venkatesh \cite[Thm.\ 1]{MichelVenkatesh07}. 
\begin{corollary}
\label{cor:main}
Let $f$ be either a holomorphic Hecke cusp form of weight $k \in 2 \Z_{>0}$ and level
$M \geq 1$ or a Hecke--Maa{\ss} cusp form of weight $0$ and level
1. Then there exists a constant $c>0$ (depending only on $M$) such that as $X\rightarrow \infty$: 
$$ \frac{|\{D\in \mathcal{D}_\mathrm{fund}^+: \epsilon_D\leq X, \exists \chi\in \widehat{\Cl_D^+}\text{ s.t. }L(f\otimes \theta_\chi,1/2)\neq 0\}|}{|\{D\in \mathcal{D}_\mathrm{fund}^+: \epsilon_D\leq X\}|}\geq c+o(1),$$
where $\mathcal{D}_\mathrm{fund}^+$ denotes the set of positive fundamental discriminants and $\epsilon_D=u+\sqrt{D}v$ denotes the fundamental positive 
 unit of discriminant $D$ (i.e. $u^2-Dv^2=1$ is the fundamental solution with $u,v\in \frac{1}{2}\Z_{\geq 1}$).
\end{corollary}
\begin{remark}
    An interesting feature of our methods is that equidistribution does indeed play a key role exactly as in the arguments of Michel--Venkatesh: the proofs of Theorems \ref{main thm nonvanishing}, \ref{thm:main} and \ref{thm:Eis} rely crucially on equidistribution of sparse subcollections of closed geodesics, see Theorems  \ref{thm:sparse} and \ref{thm:AkaEin}.
\end{remark}
\subsection{Towards normal distribution}
Theorem \ref{thm:main} implies that for Maa{\ss} cusp forms, the geodesic periods associated to $\CC$ are usually of size $ \ell( \CC)^{1/2}$. A related phenomenon has previously been observed in the context of vertical periods (e.g. modular symbols) of automorphic forms \cite{peri18} and the two phenomena are intimately related as we will see. 

Let $\Gamma$ be a Fuchsian group of finite covolume with a cusp at infinity of width one and denote  by $\Gamma_\infty=\{\begin{psmallmatrix}
    1& \Z\\ 0& 1
\end{psmallmatrix}\}$ the stabilizer of $\infty$. Let $f$ be either a Maa{\ss} cusp form of weight $0$ with Laplace eigenvalue $\lambda_f=1/4+t_f^2$ or such that $y^{-k/2}f$ is a holomorphic cusp form of weight $k\in 2\Z_{>0}$. For a double coset $[\gamma]_{\infty}:=\Gamma_\infty\gamma \Gamma_\infty \in \Gamma_{\infty} \backslash \Gamma / \Gamma_{\infty}$, we associate the vertical line integral:
$$[\gamma]_{\infty} \mapsto \int_{0}^{\infty} f(\gamma\infty+iy) \frac{dy}{y} = : L_f(\gamma {\infty}).$$
Let $c_{\gamma}$ denote the lower-left entry of the matrix $\gamma$ (note that $c_{\gamma}$ is a well-defined invariant of the class $[\gamma]_{\infty}$). For $N\geq 1$, we put
\begin{equation}\label{eq:TgammaX}
    T_\Gamma(N):=\{ [\gamma]_{\infty} \in \Gamma_{\infty} \backslash \Gamma / \Gamma_{\infty} : 0< |c_{\gamma }| \leq N  \}.\end{equation}
The distribution of  $\{ L_f (\gamma\infty) : [\gamma]_{\infty} \in T_\Gamma(N) \}$ has been extensively studied, it is known in many cases to obey asymptotically a normal distribution with variance $c_f^2\log N$ as $N \to \infty$, for some explicit $c_f>0$, see \cite{BD}, \cite{DrNo}, \cite{peri18}, \cite{Co}, \cite{No21}, \cite{LeeSun}. We emphasize that there is a structural difference between the holomorphic and non-holomorphic case: a (normalized) Hecke--Maa{\ss} cusp form of weight $0$ for $\PSL_2(\Z)$ is either a real or a totally imaginary valued function, meaning that the vertical periods follow a one-dimensional normal distribution, whereas holomorphic forms follow a two-dimensional normal distribution. 

In this paper, we introduce a method to ``lift'' the distribution of the vertical periods to study the closed geodesic periods. In particular, this means that our methods rely crucially on the existence of a cusp for $\Gamma$. We do this by investigating the relation between double cosets $\Gamma_{\infty} \backslash \Gamma / \Gamma_{\infty}$ and hyperbolic conjugacy classes of $\Gamma$ via the study of a particular graph. Recall that to each oriented closed geodesic $\CC$ on $\mathcal{X}_\Gamma$ corresponds a primitive hyperbolic conjugacy class $\{\gamma\}$ of $\Gamma$ such that the length $\ell(\CC)$ is up to a small error given by $2\log |\tr(\gamma)|$ (see eq.\ (\ref{eq:elltrace})). We define
\begin{equation}
\label{eq:tildeCgammaN}\tilde{C}_\Gamma(N):= \{ \{ \gamma \} \in \mathrm{Conj}(\Gamma) \text{ primitive}: N/2 \leq  |\tr (\gamma)| \leq N\},
\end{equation}
which is essentially $C_\Gamma(2\log N)\setminus C_\Gamma(\log N)$ in terms of equation (\ref{eq:CgammaX}). We consider the bipartite graph $G_N$ with the two vertex sets given by $T_\Gamma(N)$ and $\tilde{C}_\Gamma(N)$, and an edge between $[\gamma]_{\infty} \in T_\Gamma(N)$ and $\{ \gamma \} \in \tilde{C}_\Gamma(N)$ if and only if $[\gamma]_{\infty} \cap \{ \gamma \} \neq \emptyset$. We define a natural discrete probability measure $\mu_{G_N}$ on $\tilde{C}_\Gamma(N)$ given by the push-forward of the counting measure on $T_\Gamma(N)$ by the graph $G_N$. This is precisely defined in equation (\ref{eq:Gtransform}), see Section \ref{sec:graphs} for more details. We show that in certain cases the geodesic periods become normally distributed when counted with this measure with explicit variance given in terms of
\begin{equation}\label{eq:variance}
    (c_f)^2:=\frac{\langle f,f\rangle}{\vol(\Gamma\backslash \H)} \cdot \begin{cases}
       2^{k/2}(4\pi)^k\Gamma(\tfrac{k}{2})/\Gamma(k),& y^{-k/2}f\text{ holomorphic},\\
        \sqrt{\pi}/(\Gamma(\tfrac{3}{4}+\tfrac{1}{2}it_f)\Gamma(\tfrac{3}{4}-\tfrac{1}{2}it_f) ),&  f\text{ Maa{\ss} weight $0$}.
    \end{cases}
\end{equation}
\begin{theorem}\label{thm:CLTintro}
    Let $f$ be a real valued Hecke--Maa{\ss} cusp form of weight $0$ for $\PSL_2(\Z)$. Then for any real numbers $a<b$,
    $$ \lim_{N \to \infty} \mu_{G_N} \lr{\left \{\{\gamma\}\in \tilde{C}_\Gamma(N) : \frac{\mathcal{P}_f(\{\gamma\})}{c_f \sqrt{\log N}} \in [a,b] \right \}} = \frac{1}{\sqrt{2 \pi}} \int_a^b e^{-\frac{x^2}{2}}dx.$$
    
    Similarly, if $y^{-k/2}f$ is a holomorphic cusp form of even weight for a Fuchsian group $\Gamma$ of finite covolume with a cusp, then for any rectangle $\mathcal{R}\subset \C$,
    $$ \lim_{N \to \infty} \mu_{G_N} \lr{\left \{\{\gamma\}\in \tilde{C}_\Gamma(N) : \frac{\mathcal{P}_f(\{\gamma\})}{c_f \sqrt{\log N}} \in \mathcal{R} \right \}} = \frac{1}{2 \pi} \int_{x + iy \in \mathcal{R}} e^{-\frac{x^2+y^2}{2}}dxdy.$$
    Here $\mathcal{P}_f(\{\gamma\})=\mathcal{P}_f(\CC)$ denotes the geodesic period associated to the closed geodesic $\CC$ corresponding to the hyperbolic  conjugacy class $\{\gamma\}\in \mathrm{Conj}(\Gamma)$. 
\end{theorem}
In other words, geodesic periods are normally distributed if we count them with weights inferred by the graph $G_N$. We expect that the normal distribution should hold without these additional weights, i.e. for the counting measure, with the same variance and effective rate of convergence.

\begin{conjecture}
Let $f$ be a Hecke--Maa{\ss} cusp form of weight $0$ for $\PSL_2(\Z)$ normalized so that it takes real values. Then for any real numbers $a<b$, \begin{align*}
    \frac{1}{|\tilde{C}_\Gamma( N)|}\left|\left \{ \{ \gamma \} \in \tilde{C}_\Gamma(N): \frac{\mathcal{P}_f(\{\gamma\})}{c_f \sqrt{\log N}} \in [a,b] \right \}\right|= \frac{1}{\sqrt{2\pi}}\int_a^b e^{-\frac{x^2}{2}}dx + O_f \lr{\frac{1}{\sqrt{\log N}}}.
\end{align*}
Similarly, if $y^{-k/2}f$ is a holomorphic cusp form of weight $k\in 2\Z$ for 
 a Fuchsian group $\Gamma$ of finite covolume, then for any rectangle $\mathcal{R} \subset \C$,
\begin{align*}
    &\frac{1}{|\tilde{C}_\Gamma( N)|}\left|\left \{ \{ \gamma \} \in \tilde{C}_\Gamma(N): \frac{\mathcal{P}_f(\{\gamma\})}{c_f \sqrt{\log N}} \in \mathcal{R} \right \}\right|\\
    &= \frac{1}{2\pi}\int_{x + iy \in \mathcal{R}} e^{-\frac{x^2+y^2}{2}}dxdy + O_f \lr{\frac{1}{\sqrt{\log N}}}.
\end{align*}
\end{conjecture}
Note that by Theorem \ref{thm:main}  most geodesic periods have size comparable to the standard deviation, which is consistent with the above normal distribution conjecture.

\begin{remark}
    As alluded to above, in the case of weight 2 holomorphic forms with $\Gamma$ cocompact, a weaker version of this conjecture (without the explicit rate of convergence) was proved by Petridis--Risager \cite{PeRi2} by studying perturbations of the Selberg trace formula. In forthcoming work, we have proved a refinement (local limit laws) of the above conjecture for weight 2 holomorphic forms, following methods from perturbation theory as in \cite{Co}.
\end{remark}

\begin{remark}
    Note that the rate of convergence in this conjecture implies the size of the vanishing set is $\ll \frac{N^2}{(\log N)^{3/2}}$ which improves upon Theorem \ref{main thm nonvanishing} (with $X=2\log N$).
\end{remark}

\subsection{Idea of proof}\label{sec:ideaofproof} As mentioned above, our main idea is to transfer the distribution (in particular, the non-vanishing) of vertical periods to geodesic periods. To illustrate this, it is useful to consider where $f$ is a holomorphic cusp form of weight 2. Then $f(z) dz$ is a $\Gamma$-invariant harmonic 1-form. Hence, the following period, known as a \emph{modular symbol},
$$\inprod{\gamma}{f} : = \int_{w}^{\gamma w} f(z) dz$$
does not depend on the base point $w$ (which could be $\infty$). Moreover, the map $\inprod{\cdot}{f}: \Gamma \to \C$ is additive. This implies that $\inprod{\gamma}{f}$ is a well-defined invariant of both $[\gamma]_{\infty}$ and $\{ \gamma \}$, and so $\inprod{\gamma}{f}$ is constant along the connected components of the graphs $G_N$ defined above. So in this case there is a straightforward connection between vertical line periods and geodesic periods.

For automorphic forms of general weight, this is no longer true. We use instead a key ingredient from \cite{no23}, which allows us to find a relation between the vertical line period and the geodesic period corresponding to an edge in the bipartite graph $G_N$ defined above. We show that if $f$ is a cusp form of weight $k\in 2\Z$ and $\gamma \in\Gamma$ is hyperbolic with lower-left entry $c_{\gamma}>0$, then 
\begin{equation*}
    \mathcal{P}_f(\{ \gamma \}) = (-1)^{k/2+1} L_f(\gamma \infty) +O_{f,\varepsilon}\left(1+\left(\frac{c_\gamma}{|\tr (\gamma)|}\right)^{1/2+\varepsilon}\right),
\end{equation*}
and a similar statement holds for Eisenstein series, see Proposition \ref{prop:main}.

In particular, we obtain that for $100\%$ of edges, the vanishing of the geodesic period implies that the vertical period is very small. 
In the cases where the normal distribution of vertical periods is known, we have that almost all of the vertical periods are large (of size $\sqrt{\log N}$). This yields the wanted conclusion if we can bound from above the degrees in $G_N$ of vertices in $T_\Gamma(N)$ and from below the degrees of vertices in $\tilde{C}_\Gamma(N)$. 

One can easily obtain good estimates for the degrees of the vertices in $T_\Gamma(N)$ from the basic properties of the graph by counting matrices with fixed lower left entry and bounded trace, see Lemma \ref{lem:c_x}. Using a geometric argument, we show that the degree of a vertex in $\tilde{C}_\Gamma(N)$ is lower bounded by the length of the corresponding geodesics restricted to a subdomain of the fundamental domain for $\Gamma \backslash \H$: there exists a region $\mathcal{B} \subset \Gamma \backslash \H$ such that
\begin{equation}\label{eq:introdeg}\deg (\{ \gamma \}) \gg l (\CC \cap \mathcal{B}),\end{equation}
see Proposition \ref{prop:degy}. 
By combining ergodicity of the geodesic flow with the equidistribution of closed geodesics of bounded length \cite[Cor.\ 6.5]{zel92}, we obtain  equidistribution for sparse subcollections of closed geodesics, see Theorem \ref{thm:sparse}, thus obtaining lower bounds for the right-hand side of (\ref{eq:introdeg}) on average. In the case of the modular group, we obtain equidistribution for even sparser subcollections by a result of Aka--Einsiedler \cite[Thm.\ 2]{akaeins}, see Theorem \ref{thm:AkaEin}. Combining these ingredients yields the proofs of Theorems \ref{main thm nonvanishing}, \ref{thm:main} and \ref{thm:Eis}. The proof of Theorem \ref{thm:CLTintro} follows by a general argument which allows one to ``lift'' the distribution of measures from one component to the other in a bipartite graph, as developed in Section \ref{sec:normal}.  


\begin{remark}
    We obtain stronger non-vanishing results for the modular group $\Gamma = \PSL_2 (\Z)$ and the reason is twofold. Firstly, in this case, from \cite{BD} and \cite{DrNo}, we know precise rate of convergence towards the normal distribution of the set $\{ L_f(\gamma \infty) : [\gamma]_{\infty} \in T_\Gamma(N) \}$, and hence we can deduce better upper bounds for the subset of cosets with small vertical line period. Secondly, in the arithmetic setting, it is possible to obtain equidistribution for a sparser subcollection of closed geodesics, as in \cite{akaeins}.  
\end{remark}
\begin{remark}
    We note that if a central limit theorem was known for vertical line periods of automorphic forms on a general $\Gamma$ with a cusp, our methods allow us to extend Theorem \ref{thm:main} to any such automorphic form. Similarly, one could extend Theorem \ref{main thm nonvanishing} to any arithmetic subgroup with a cusp. Such a central limit theorem has been announced by Bettin--Drappeau--Lee.  
\end{remark}
\begin{remark}
 It is natural to ask what is the more precise relationship  between the measures $\mu_{G_N}$ and the uniform distributions $\mu_N$ on $\tilde{C}_\Gamma(N)$. The argument sketched above gives the following: there exists a constant $c=c(\Gamma)>0$ such that for any $\varepsilon>0$  and any sequence of subsets $A_N\subset \tilde{C}_\Gamma(N) $  satisfying $\mu_N(A_N)\geq \varepsilon$ it holds that 
 $ \mu_{G_N}(A_N)\geq c \varepsilon^2 $ for $N$ large enough. It would be interesting to see if one can prove the stronger statement that there exists $0<c_1<c_2$ such that $c_1\varepsilon\leq \mu_{G_N}(A_N)\leq  c_2\varepsilon$ for $N$ large enough. This would require obtaining stronger upper and lower bounds on the degrees in the graphs $G_N$ of the vertices in $\tilde{C}_\Gamma(N)$.  
\end{remark}

\subsection{Structure of the paper}
In Section \ref{sec:MicVen} we briefly discuss the work of Michel-Venkatesh for the imaginary quadratic case and why their method falls short in the real quadratic case.\\
In Section \ref{sec:background} we introduce the required background material, including the connection between vertical line integral and geodesic periods from \cite{no23}.\\
In Section \ref{sec:graph specifics} we define the bipartite graph $G_N$ alluded above and prove the key bounds on the degrees in this graph. This includes a geometric argument relying on the equidistribution of sparse collections of closed geodedics as proved in Section \ref{sec:sparse}.\\
In Section \ref{sec:nonvanishgeo} we use the results in the preceding section to complete the proofs of our main Theorems \ref{main thm nonvanishing} and \ref{thm:main}.\\
In Section \ref{sec:normal} we prove a general result on how to ``lift'' the distribution from one component of a bipartite graph to the other. We use this to prove Theorem \ref{thm:CLTintro}.\\
In Section \ref{sec:application} we obtain the application towards non-vanishing of central values of Rankin-Selberg $L$-functions.

\section*{Acknowledgemnts}
The authors would like to thank Philippe Michel, Farrell Brumley, Morten Risager and Sary Drappeau for useful discussions and comments, as well as the referees for valuable suggestions which improved the quality of the paper significantly. This material is based upon work supported by the Swedish Research Council under grant no. 2021-06594 while the authors were in residence at Institut Mittag-Leffler in Djursholm, Sweden during the  Analytic Number Theory programme in the spring of 2024.

\section{Equidistirbution of Heegner points and non-vanishing of Rankin--Selberg $L$-functions, following Michel--Venkatesh}\label{sec:MicVen}
In an elegant paper \cite{MichelVenkatesh07} Michel--Venkatesh introduced a new method for evaluating the first moment of certain Rankin--Selberg $L$-functions associated to theta series of class group characters of imaginary quadratic fields. Their approach used equidistribution of Heegner points combined with Waldspurger's formula, the Plancherel formula and subconvexity estimates to obtain quantitative non-vanishing results. The methods has subsequently been extended to calculating more general ``wide moments'' of $L$-functions by the second-named author \cite{Nordentoft21}, see also \cite{BurungaleHida16}.
During the problem session at the Oberwolfach workshop \cite[p.\ 1377]{Oberwolfach20} Michel asked whether one could extend this to real quadratic field under some ``non-smallness'' assumption on the class number (see Question \ref{q:michel} above). The difficulty of this problems stems from the infinitude of the unit group, which makes the original approach in \cite{MichelVenkatesh07} fall short as we will now explain. 

We start by sketching the argument from \cite{MichelVenkatesh07} in the imaginary quadratic case. Let $f$ be an even Hecke--Maa{\ss} form of weight $0$ and level $1$. Let $K/\Q$ be an imaginary quadratic field  of discriminant $D<-6$ with class group $\Cl_K$ and class number $h(D):=|\Cl_K|$. Given $A\in \Cl_K$ we denote by $z_A\in \mathcal{X}_0(1):=\PSL_2(\Z)\backslash \H$ the associated Heegner point (for a definition see e.g. \cite[p.\ 75]{Duke88}). 
Let $\chi\in \widehat{\Cl_K}$ be a class group character of $K$ and denote by $\theta_\chi$ the associated theta-series via automorphic induction. Waldspurger's formula due to Zhang \cite{Zh01}, \cite{Zh04} in its explicit form gives: 
$$L(f\otimes\theta_\chi,1/2)=\frac{c^-_f}{|D|^{1/2}}\left| \sum_{A\in \Cl_K}\chi(A) f(z_A) \right|^2,$$
for some explicit constant $c^-_f>0$.
By Plancherel (i.e. character orthogonality) and Duke's equidistribution theorem for Heegner points \cite{Duke88} it follows that as $D\rightarrow -\infty$ we have
\begin{equation}\label{eq:immoment}\frac{1}{h(D)}\sum_{\chi\in \widehat{\Cl_K}} L(f\otimes\theta_\chi,1/2)=\frac{c_f^-}{|D|^{1/2}} \sum_{A\in \Cl_K} |f(z_A)|^2=\frac{c_f^- h(D)}{|D|^{1/2}}(|\!| f|\!|_{L^2}+o(1)).\end{equation}
In particular, the left-hand side is non-vanishing for $|D|$ sufficiently large. Furthermore, by applying the subconvexity bound for the Rankin--Selberg $L$-functions due to Harcos--Michel \cite{HarcosMichel06} one obtains a non-vanishing proportion of $\gg |D|^\delta$ for some (small) $\delta>0$. 

Now let $K$ be a real quadratic field of discriminant $D>0$ with narrow class group $\Cl_K^+$ and wide class group $\Cl_K$. Denote the narrow class number by $h^+(D):=|\Cl_K^+|$. In this setting we can associate to an ideal class $A\in \Cl_K^+$ a primitive oriented closed geodesics $\CC_A$ on $\mathcal{X}_0(1)$ analogues to the correspondence between Heegner points and ideal classes above. 
Let $f$ be Hecke--Maa{\ss} form of weight $0$ and level $1$ and assume for simplicity that $f$ is even. Similarly to the imaginary case, by employing  the formula (\ref{eq:popa}) due to Popa \cite[Thm.\ 1]{popa06} and orthogonality of characters we obtain: 
\begin{equation}\label{eq:realmoment}\frac{1}{h^+(D)}\sum_{\chi\in\widehat{\Cl_K^+}:\chi(J)=1}L(f\otimes\theta_\chi,1/2)=\frac{c^+_f}{D^{1/2}} \sum_{A\in \Cl^+_K} \left|\int_{\mathcal{C}_A}f(z)\tfrac{|dz|}{\Im z} \right|^2.\end{equation}
We see however that equidistribution of the geodesics does \emph{not} imply non-vanishing of the right-hand side due to the fact that the square is on the ``outside'' of the integral (over the closed geodesics). To fix this, one has to allow for non-trivial infinity type and consider the Arakelov class group of $K$:
$$ \Cl_K^\mathrm{Ara}:= K^\times \A_\Q^\times \backslash \A_K^\times/\widehat{\mathcal{O}}_K^\times\cong \Cl^+_K\times \R_{>0} /\epsilon_K^\Z,$$
where $\epsilon_K>1$ is the fundamental positive unit of $K$. Given a character $\chi\in \widehat{\Cl_K^\mathrm{Ara}}$, the infinity type is parameterized by a number $\lambda_\chi\in \frac{1}{\log \epsilon_K}\Z$. As explained in \cite[Sec. 4.2] {BlomerBrumley20} using equation (4.7) in \emph{loc.\ cit.} one has:
$$ \frac{1}{h^+(D)}\sum_{\chi\in\widehat{\Cl^\mathrm{Ara}_K}}L(f\otimes\theta_\chi,1/2)\psi_f(\lambda_\chi)= \frac{c^+_f}{D^{1/2}} \sum_{A\in \Cl_K^+} \int_{\mathcal{C}_A}|f(z)|^2\tfrac{|dz|}{\Im z},$$
where $\psi_f(\lambda_\chi)$ denotes a weight function satisfying, by \cite[(4.8)] {BlomerBrumley20}, the following bound:
$$\psi_f(\lambda_\chi)\ll e^{-c_0|\lambda_\chi|/|\lambda_f|},$$
for some absolute constant  $c_0> 0$.
Notice that the equidistribution theorem of Duke yields that indeed the left-hand side is non-zero for $D$ large enough. Thus we obtain non-vanishing Rankin-Selberg $L$-functions with $\chi$ an Arakelov class group character. Because of the rapid decay of $\psi_f(\lambda_\chi)$ and the convexity bound for the $L$-values, we can ensure that $|\lambda_\chi|\leq c_1\log D$ for some sufficiently large $c_1>0$. Notice that the family of such characters is of size $D^{1/2+o(1)}$, exactly as in the imaginary quadratic case. In other words, the question of Michel amounts to whether one can ensure that the  infinity type is trivial. 
\begin{remark}
    In \cite{MichelVenkatesh07}, similar results are obtained for holomorphic forms of weight $2$ using the Jacquet--Langlands correspondence and Waldspurger's formula for definite quaternion algebras due to Gross in this case. 
\end{remark}
\begin{remark}
Templier \cite{Templier11rank} extended the approach of Michel--Venkatesh to derivatives using the Gross--Zagier formula. Later Templier \cite{Templier11} and Templier--Tsimerman \cite{TemplierTsimerman13} evaluated the left-hand side of (\ref{eq:immoment}) using tools from analytic number theory (approximate functional equation and (non-split) shifted convolution sums). See also \cite{HuangLester23}. It would be interesting to see whether a variation of these analytic methods can be made to work in the real quadratic case. 
\end{remark}

\section{Background}\label{sec:background}

\subsection{Fuchsian groups}
Let $\Gamma < \PSL_2(\R)$ be a Fuchsian group of finite covolume acting on hyperbolic $2$-space $\H:=\{z\in \C: \Im z>0\}$ via M\"{o}bius transformations. Consider the associated hyperbolic $2$-orbifold 
$$\mathcal{X}_{\Gamma}:= \Gamma \backslash \H,$$ 
equipped with the hyperbolic volume element $\tfrac{dxdy}{y^2}$ and line element $\tfrac{|dz|}{\Im z}$. It is a general fact that there is a one-to-one correspondence between primitive closed geodesics on $\mathcal{X}_{\Gamma}$ and primitive conjugacy classes of hyperbolic elements of $\Gamma$. We denote by $\mathcal{C}_{\gamma}$ the oriented closed geodesic on $\mathcal{X}_\Gamma$ corresponding to a hyperbolic conjugacy class $\{ \gamma \}\in \mathrm{Conj}(\Gamma)$.


Each hyperbolic $\gamma \in \Gamma$ is conjugate in $\PSL_2(\R)$ with some $\begin{psmallmatrix}
    t & 0 \\ 0 & t^{-1} \end{psmallmatrix}$ with $t>1$. The norm of $\gamma$ is $N(\gamma) = t^2$, the trace is $\tr(\gamma) = t+t^{-1}$ and the length of the corresponding geodesic is \begin{equation}\label{eq:elltrace}
        \ell (\CC_{\gamma}) := \int_{\CC_\gamma} 1 \tfrac{|dz|}{\Im z}= \log N(\gamma) = 2 \log t=2\log \tr(\gamma)+O((\tr(\gamma))^{-2}).\end{equation}

Given a collection $I$ of closed geodesics on $\mathcal{X}_\Gamma$, we denote the total length by
\begin{equation}
    \label{eq:collection} \ell(I):=\sum_{\CC\in I}\ell(\CC).
\end{equation}
The unit tangent bundle of $\mathcal{X}_\Gamma$ admits a natural description as a homogeneous space: 
$$\mathbf{T}^1 (\mathcal{X}_{\Gamma}) \simeq \Gamma \backslash \PSL_2(\R),$$
recalling that $\H \simeq \PSL_2(\R) / \mathrm{PSO}_2$, equipped with Haar probability measure $dg$ and the associated inner product $\langle \cdot, \cdot \rangle$. In particular, the unit tangent bundle $\mathbf{T}^1 (\mathcal{X}_{\Gamma})$ admits a right action of $\PSL_2(\R)$ given by $\Gamma x.g = \Gamma x g$. The diagonal subgroup 
$$A:=\{ a_t : t \in \R\} \leq \PSL_2(\R),\quad a_t := \begin{pmatrix}
    e^{t/2} & 0 \\ 0 & e^{-t/2}
\end{pmatrix},$$
generates the geodesic flow on $\mathbf{T}^1 (\mathcal{X}_{\Gamma})$. Let $\CC \subset \mathcal{X}_{\Gamma}$ be an oriented closed geodesic and consider its lift to the unit tangent bundle. By a slight abuse of notation, we denote the lift by the same symbol $\CC \subset \mathbf{T}^1 (\mathcal{X}_{\Gamma})$. This yields a one-to-one correspondence between oriented closed geodesic on $\mathcal{X}_\Gamma$ and  closed and compact $A$-orbits in $\mathbf{T}^1 (\mathcal{X}_\Gamma)$. We denote by $\mu_{\CC}$ the unique $A$-invariant measure on $\CC$ with volume $\ell(\CC)$. For $F:\Gamma \backslash \PSL_2(\R) \to \C$ integrable with respect to $\mu_{\CC}$, we have explicitly  
\begin{align*}
    \int_{\CC} F(g) d \mu_{\CC}(g) = \int_0^{\ell(\CC)} F( xa_t) dt, \quad \text{for any } x \in \CC.
\end{align*}
We refer to \cite[Chapter 9]{EW} for more details about the geodesic flow on quotient surfaces.

\subsubsection{Automorphic forms} We will now recall some standard facts about the spectral theory of automorphic forms, we refer to \cite[Sec. 4]{DuFrIw02} for further background.  Let $f: \H \to \C$ be an \emph{automorphic function of weight $k$ for $\Gamma$}, that is a smooth function such that
\begin{equation}
\label{eq:automorphy}f(\gamma z) = j_{\gamma}(z)^k f(z), \quad \text{for all }\gamma \in \Gamma,
\end{equation}
where $j_{\gamma}(z)=\frac{cz+d}{|cz+d|}$ with $\gamma= \begin{pmatrix}
    a & b \\ c& d
\end{pmatrix}$. We equip the space of weight $k$ automorphic functions with the \emph{Petersson inner product}
$$ \langle f,g\rangle= \int_{\mathcal{X}_\Gamma} f(z)\overline{g(z)} \frac{dxdy}{y^2}.$$
If $f$ is square integrable with respect to the Petersson inner product and is eigenfunction of the weight $k$ Laplacian
$$\Delta_k := - y^2 \lr{\pdv[2]{}{x} + \pdv[2]{}{y}} + ik y \pdv{}{x},$$
then we say that $f$ is a \emph{Maa{\ss} form of weight $k$}. If $f$ is rapidly decaying at the cusps of $\Gamma$, we say that $f$ is a \emph{Maa{\ss} cusp form of weight $k$}, which implies that $f$ is square integrable with respect to the Petersson inner product. If a Maa{\ss} form of weight $k$ is not cuspidal, we call it a \emph{residual Maa{\ss} form of weight $k$}. 
Note that if $g\in \mathcal{S}_k(\Gamma)$ is an even weight $k$ holomorphic cusp form for $\Gamma$, then $z\mapsto y^{k/2}g(z)$ defines a Maa{\ss} cusp form of weight $k$ with eigenvalue $\tfrac{k}{2}(1-\tfrac{k}{2})$. We call these automorphic forms  \emph{associated to holomorphic cusp forms}. 
We say that a Maa{\ss} form is of \emph{level $M$} if it is an automorphic form for the Hecke congruence group $\Gamma_0(M):=\{\gamma\in \PSL_2(\Z):\gamma\equiv\begin{psmallmatrix}
    \ast & \ast\\0& \ast
\end{psmallmatrix}\modulo M\}$. For a Maa{\ss} form $f$ of weight $0$ and level $M$ which is an eigenfunction for $z\mapsto -\overline{z}$, we refer to the eigenvalue $\epsilon(f)\in \{\pm 1\}$ as the \emph{sign of $f$} and say that it is \emph{even}, resp.\ \emph{odd}, see e.g. \cite[eq.\ (4.65)]{DuFrIw02}.

We will also consider certain non-square integrable automorphic forms, namely the \emph{Eisenstein series}. We refer to \cite[Sec.\ 3.2]{Iw} for details. Let $\mathfrak{a}$ be a cusp of $\Gamma$ with stabilizer $\Gamma_\mathfrak{a}\subset \Gamma$ and scaling matrix $\sigma_\mathfrak{a}$. For $k\in 2\Z$ we define:
\begin{equation}\label{eq:Eis}
    E_{\mathfrak{a},k}(z,s):=\sum_{\gamma\in \Gamma_\mathfrak{a}\backslash\Gamma} j_\gamma(z)^{-k}\Im (\sigma_\mathfrak{a}^{-1} \gamma z)^s,\quad \Re s>1
\end{equation}
and elsewhere by meromorphic continuation. Note that the unitary Eisenstein series with spectral parameter $\Re s=1/2$ are exactly the ones showing up in the spectral theorem (see \cite[Ch.\ 7]{Iw}), which we will simply refer to as \emph{Eisenstein series of weight $k$}.  In the case $\Gamma=\PSL_2(\Z)$, we further define the \emph{completed Eisenstein series of weight $k$}:
\begin{equation}
    \label{eq:completedEis}
E^\ast_{k}(z,s):=\pi^{-s}\zeta(2s)\Gamma(s+k/2)E_{\infty,k}(z,s),\end{equation}
which we note is non-zero at the central point $s=1/2$ by \cite[Sec.\ 6.4]{Iw}. 


Let $f$ be an automorphic form for $\Gamma$ of weight $k\in 2\Z$. Denote by $F$ the lift of $f$ to to $\mathbf{T}^1(\mathcal{X}_\Gamma)$, given by
$$F(g) := j_{g}(i)^{-k} f(g.i).$$
We define the \emph{geodesic period associated to $f$ and $\CC$} as
\begin{equation}\label{eq:geodesicperiods}\mathcal{P}_{f}(\CC) : = \int_{\CC} F(g) d \mu_{\CC}(g).\end{equation}
If the oriented closed geodesic $\CC$ corresponds to the hyperbolic conjugacy class $\{\gamma\}\in \mathrm{Conj}(\Gamma)$, we may write $\mathcal{P}_{f}(\{\gamma\}) = \mathcal{P}_{f}(\CC)$.

    

\subsection{From closed geodesic to vertical periods}
We fix $\Gamma\subset \PSL_2(\R)$ to be a Fuchsian group of finite covolume with a cusp at infinity of width one. We are now ready to state the key connection between vertical line integrals and geodesic periods. This was proved in \cite[Prop.\ 4.2]{no23} in a slightly different form. We will simply indicate the necessary changes.
\begin{proposition}[Cf.\ Proposition 4.2 in \cite{no23}]
\label{prop:main}
Let $k\in 2\Z$ be an even integer. Let $f:  \H\rightarrow \C$ be a Maa{\ss} cusp form of weight $k$ for $\Gamma$. Let $\gamma\in \Gamma$ be a hyperbolic matrix with lower left entry $c_\gamma>0$, and let $\CC_\gamma\subset \Gamma\backslash \PSL_2(\R)$ be  the oriented closed geodesic corresponding to the $\Gamma$-conjugacy class of $\gamma$. For any $\varepsilon>0$, we have that
\begin{align}\label{eq:cuspmain}  \mathcal{P}_f(\CC_\gamma)= (-1)^{k/2+1}\int_0^\infty f\left(\gamma \infty+iy\right)\frac{dy}{y}+O_{f,\varepsilon}\left(1+\left(\frac{c_\gamma}{|\tr (\gamma)|}\right)^{1/2+\varepsilon}\right).\end{align}

Similarly, for $f$ either an Eisenstein series or a residual Maa{\ss} form of weight $k$ with Laplace eigenvalue $1/4+t^2$, we have 
\begin{align}\label{eq:eisensteinmain}  \mathcal{P}_f(\CC_\gamma)=&(-1)^{k/2+1}\int_0^\infty \left(f\left(\gamma \infty+iy\right)-f_\infty\left(y\right)\right)\frac{dy}{y}  \\ \nonumber
&+A_f c_\gamma^{-1/2-it}\frac{1+(-1)^{k/2}}{1/2+it}+ B_f c_\gamma^{-1/2+it}\frac{1+(-1)^{k/2}}{1/2-it}\\
\nonumber &+O_{f,\varepsilon}\left(\left(\frac{c_\gamma}{|\tr (\gamma)|}\right)^{1/2+\varepsilon}+\left(\frac{c_\gamma}{|\tr (\gamma)|}\right)^{-1/2-\varepsilon}\right), \end{align}
where $f_\infty(y)=A_fy^{1/2-it}+B_fy^{1/2+it}$ denotes the constant Fourier coefficient at $\infty$.
\end{proposition}
\begin{proof}[Proof sketch] The formula in \cite[Prop 4.2]{no23} gives a relation in the case $\Gamma=\Gamma_0(N)$ between the geodesic periods of an automorphic form $f$ and additive twist $L$-series of $f$, which are certain linear combinations of vertical line integrals of $f$. On the one hand, we need to argue that the formulas (\ref{eq:cuspmain}) and (\ref{eq:eisensteinmain}) correspond exactly to the ones in \cite[Prop.\ 4.2]{no23} by rewriting the additive twists as line integrals. Secondly, we want a similar formula for general $\Gamma$. To obtain this, notice that the arguments in \cite[Sec. 4]{no23} carry over to general Fuchsian groups of finite covolume with a cusp at infinity of width one. The only difference here is that the dependence on the spectral data of $f$ might change, see Remark 4.1 in \emph{loc.\ cit.}. Now the claimed result follows by inserting equations (4.31) and (4.32) of \emph{loc.\ cit.}\ into equation (4.27) of \emph{loc.\ cit.}\ and doing the change of variables $z\mapsto \gamma z$ in the line integral in (4.32).
\end{proof} 

\subsection{Normal distribution of vertical periods}
A key input in our proofs is the normal distribution of vertical periods of automorphic forms, as explored by many authors (see below). The starting point for these results were the conjectures of Mazur--Rubin--Stein \cite{MazurRubin21} motivated by understanding the distribution of central values of the twisted $L$-functions $L(E,\chi,s)=\sum_{n\geq 1}a_E(n) \chi(n) n^{-s}$, where $E/\Q$ is an elliptic curve with Dirichlet series coefficients $a_E(n)$, and $\chi$ is a Dirichlet character modulo $q$. These twisted $L$-values can be expressed in terms of \emph{modular symbols}:
\begin{equation}
    \langle \tfrac{a}{q}\rangle_E:=2\pi i\int_{a/q}^\infty f_E(z)dz,\quad a\in (\Z/q)^\times,
\end{equation}
where $f_E$ denotes the holomorphic cusp form of weight $2$ associated to $E$ by modularity.
The \emph{Birch--Stevens formula} then gives the following formula for the central value:
\begin{equation}
   \tau(\overline{\chi})L(E,\chi,1) =\sum_{a\in (\Z/q)^\times} \overline{\chi}(a)\langle \tfrac{a}{q}\rangle_E ,
\end{equation}
for $\chi$ a primitive Dirichlet character modulo $q$, where $\tau(\overline{\chi})$ denotes the Gau{\ss} sum. Based on numerical experiments, Mazur--Rubin--Stein conjectured that 
$$\{\Re (\langle \tfrac{a}{q}\rangle_E)/\sqrt{\log q}: a\in (\Z/q)^\times\},$$ is asymptotically normally distributed as $q\rightarrow \infty$ \cite[Conj.\ 4.3]{MazurRubin21}. This conjecture was settled on average over $q$ by Petridis--Risager \cite{peri18}. 

Now let $f$ be an arbitrary automorphic form with respect to a Fuchsian group $\Gamma$ of finite covolume with a cusp at $\infty$.  Then the  (averaged version of the) conjecture of Mazur--Rubin--Stein admits a natural generalization in terms of the vertical periods: 
\begin{align} \label{eq:verticalperiods}
\int_0^\infty (f(\gamma \infty+iy)-f_\infty(y))\frac{dy}{y},\quad \gamma \in \Gamma 
\end{align}
where $f_\infty(y)$ denotes the constant term of $f$ at the cusp $\infty$. Note that $f_{\infty}(y)=0$ if $f$ is a cusp form. We note that in the case where $\Gamma$ is a congruence subgroup and $f$ is a Hecke eigenform, then the vertical periods (\ref{eq:verticalperiods}) also satisfy a Birch--Stevens type formula, see \cite[Sec.\ 8.1]{DrNo}. 

The distribution of the more general vertical periods (\ref{eq:verticalperiods}) have been studied by many authors using diverse tools \cite{BD}, \cite{DrNo}, \cite{peri18}, \cite{Co}, \cite{No21}, \cite{LeeSun}. In all known cases, the limiting distribution is that of a Gau{\ss}ian. Note that these results are all phrased in terms of \emph{additive twist $L$-series} \cite[Eq.\ (5.1)-(5.2)]{DrNo}. So in order to apply the results in \emph{loc.\ cits.} and pin down the variance as in (\ref{eq:variance}), we need to rewrite the vertical integrals (\ref{eq:verticalperiods}) in terms of these $L$-series. If $f$ is associated to a holomorphic cusp form of even weight $k\geq 2$, we have 
\begin{equation}\label{eq:holomorphicperiods}
    \int_0^\infty f(\gamma \infty+iy)\frac{dy}{y}=2^{k/2}\Gamma(\tfrac{k}{2}) L(f,\gamma \infty,1/2),
\end{equation}
where 
$$L(f,x, s):=\sum_{n\geq 1}a_f(n)e(nx)n^{1/2-s},$$ for $\Re(s)>1$ and $x \in \R$, and it admits analytic continuation to all of $\C$ when $x$ is $\Gamma$-equivalent to a cusp, see \cite[Lemma 5.3]{DrNo}. Here $a_f(n)$ denote the Fourier coefficients of $f$ normalized as in \cite[Eq.\ (2.4)]{DrNo}.

If $f$ is a Hecke--Maa{\ss} cusp form of level $1$, weight $0$ and sign $\epsilon(f)=\pm 1$ or a completed Eisenstein series of level $1$ and weight $0$ (in which case $\epsilon(f)=1$), then
\begin{equation}\label{eq:verticalperiodtoadditive}
    \int_0^\infty (f(\gamma\infty+iy)-f_\infty(y))\frac{dy}{y}=  \frac{2\pi^{3/2}}{\cosh(\pi t_f)\Gamma(\frac{3}{4}+\frac{1}{2}it_f)\Gamma(\frac{3}{4}-\frac{1}{2}it_f)}L^\pm(f,\gamma \infty, 1/2), 
\end{equation}
where 
$$L^\pm(f,x, s):=\sum_{n\geq 1}a_f(n)n^{1/2-s}\begin{cases}
    \cos (nx),& \pm=+,\\
    i\sin (nx),& \pm=-
    ,
\end{cases}$$
for $x\in \Q$, $\Re s>1$ and elsewhere by analytic continuation,
see \cite[Lemma 5.3]{DrNo}. 

The central limit theorem for additive twist $L$-series in the case of holomorphic forms with effective rate of convergence has been obtained by, in certain cases respectively, by the first named author \cite[Thm.\ 1.4]{Co} using perturbation theory, and by Bettin and Drappeau  \cite[Cor.\ 1.5]{BD} using dynamics of the Gau{\ss} map.  Combing this with the formula (\ref{eq:holomorphicperiods}) we arrive at the following result for the vertical periods. Recall here the definition (\ref{eq:TgammaX}) of the set $T_\Gamma(N)$ of double cosets with lower left entry bounded by $N$.
\begin{theorem}[C., Bettin--Drappeau]\label{thm:normaleffective1}
    Let $f$ be associated to an even weight holomorphic cusp form for $\Gamma=\PSL_2(\Z)$, or to a weight 2 holomorphic form for a finite covolume Fuchsian group $\Gamma$ with a cusp at $\infty$ of width one.  Then for any rectangle $\mathcal{R}\subset \C$, 
    \begin{align*}  &\frac{1}{|T_\Gamma(N)|}\left|\left\{ [\gamma]_\infty \in T_\Gamma(N) : \frac{\int_{0}^\infty f(\gamma \infty +iy)\frac{dy}{y} }{c_f\sqrt{\log N}} \in \mathcal{R}\right\}\right|  \\
    \nonumber &= \frac{1}{2 \pi} \int_{x+iy\in \mathcal{R}} e^{-\frac{x^2+y^2}{2}}dxdy + O_f \lr{\frac{1}{\sqrt{\log N}}}.\end{align*}
 The implied constant may depend on $f$ but is independent of $\mathcal{R}$.
\end{theorem}

For holomorphic forms of general weight $k$ with respect to a general Fuchsian group of finite covolume with a cusp, the normal distribution has been obtained by the second named author \cite[Thm.\ 5.1]{No21} using techniques from spectral theory. In this setting the rate of convergence is not known.

\begin{theorem}[N.]\label{thm:normalnoneffective}
    Let $f$ be associated to a holomorphic cusp form for a Fuchsian group $\Gamma$ of finite covolume with a cusp at infinity of width one. Then for any rectangle $\mathcal{R}\subset \C$,
   \begin{align*}  \lim_{N \to \infty} \frac{1}{|T_{\Gamma}(N)|}\left|\left\{ [\gamma]_\infty \in T_\Gamma(N) : \frac{\int_{0}^\infty f(\gamma \infty +iy)\frac{dy}{y} }{c_f\sqrt{\log N}} \in \mathcal{R}\right\}\right| 
   = \frac{1}{2 \pi} \int_{x+iy\in \mathcal{R}} e^{-\frac{x^2+y^2}{2}} dxdy.\end{align*}
\end{theorem}

Drappeau and the second named author \cite[Thm.\ 1.5]{DrNo} extended the result of Bettin--Drappeau to general Maa{\ss} forms of level $1$ (we notice that the effective rate of convergence is not explicitly stated in \emph{loc.\ cit.} but follows as in \cite[p.\ 1412]{BD} using \cite[
Prop.\ 7.2]{DrNo} and the Berry--Esseen inequality). By (\ref{eq:verticalperiodtoadditive}) and \cite[Cor.\ 7.7]{DrNo} (recall here from \cite[Eq.\ (5.101)]{IwKo} that $L(\sym^2 f,1)=\langle f,f\rangle \cosh (\pi t_f)/(2\pi^2/3)$) we get a one dimensional normal distribution result.  
\begin{theorem}[Drappeau--N.]\label{thm:normaleffective2}
    Let $f$ be a Hecke--Maa{\ss} cusp form of weight $0$ for $\Gamma=\PSL_2(\Z)$ normalized so that it takes real values.  Then for any real numbers $a<b$, 
    \begin{align*}  &\frac{1}{|T_\Gamma(N)|}\left|\left\{ [\gamma]_\infty \in T_\Gamma(N) : \frac{\int_{0}^\infty f(\gamma \infty +iy)\frac{dy}{y} }{c_f\sqrt{\log N}} \in [a,b]\right\}\right|  \\
    \nonumber &= \frac{1}{\sqrt{2 \pi}} \int_a^b e^{-\frac{x^2}{2}}dx + O_f \lr{\frac{1}{\sqrt{\log N}}}.\end{align*}
 The implied constant may depend on $f$ but is independent of $a,b$.
\end{theorem}

The normal distribution of vertical periods of the completed Eisenstein  series for $\Gamma=\PSL_2(\Z)$ at $s=1/2$ was achieved by Bettin--Drappeau \cite[Thm.\ 2.1]{BD}. In this case the variance is given by $(c_{E^\ast})^2:=\sqrt{\pi}\Gamma(\tfrac{3}{4})^{-2}$ (again translating between the Estermann function and the vertical period using (\ref{eq:verticalperiodtoadditive})).
\begin{theorem}[Bettin--Drappeau]\label{thm:normalnoneffective2}
    Let $E^\ast$ denote the completed Eisenstein series of weight $0$ for $\Gamma=\PSL_2(\Z)$ and spectral parameter $s=1/2$ as in (\ref{eq:completedEis}). Let $\varepsilon>0$. Then for any real numbers $a<b$,
   \begin{align*}  &\frac{1}{|T_\Gamma(N)|}\left|\left\{ [\gamma]_\infty \in T_\Gamma(N) : \frac{\int_{0}^\infty (E^\ast(\gamma \infty +iy)-E^\ast_\infty(y))\frac{dy}{y} }{c_{E^\ast}\sqrt{\log N}} \in [a,b]\right\}\right|  \\ 
   &= \frac{1}{\sqrt{2 \pi}} \int_{a}^b e^{-\frac{x^2}{2}} dx+O_\varepsilon\left(\frac{1}{(\log\log N)^{1-\epsilon}}\right).\end{align*}
 The implied constant may depend on both $\varepsilon$ but is independent of $a,b$.
\end{theorem}


\section{Bipartite graphs on double cosets and conjugacy classes}
\label{sec:graph specifics}
In this section we prove the necessary results required for the proofs of our main Theorems \ref{main thm nonvanishing} \ref{thm:main} and  \ref{thm:Eis}. We start by defining the bipartite graph $G_N$, key to our approach alluded to in the introduction. This graph has as its two components double cosets with lower left entry bounded by $N$ and conjugacy classes with trace between $N/2$ and $N$, respectively, with an edge if the conjugacy class and double coset intersects. The key properties of this graph are bounds on the degrees of the vertices. For one component this is an elementary matrix count and for the other we  achieve it by a geometric argument relying on equidistribution of sparse collections of closed geodesics. 

\subsection{Bipartite graphs}Let $X$ and $Y$ be finite sets. We define a \emph{bipartite graph on $X,Y$} as a subset of the product
$$G\subset X\times Y.$$
We define \emph{the neighbors of $x\in X$ (resp.\ $y\in Y$)} as 
$$e(x):=\{y\in Y: (x,y)\in G\}\subset Y,\quad e(y):=\{x\in X: (x,y)\in G\}\subset X,$$
and for a subset $A\subset X$ (resp.\ $B\subset Y$) we denote
$$ e(A)=\bigcup_{x\in A}e(x)\subset  Y,\quad (\text{resp. }\, e(B)=\bigcup_{y\in B}e(y)\subset X). $$
We define \emph{the degree of $x$ (resp.\ $y$)} as  
$$\deg(x):=\#e(x),\quad (\text{resp. }\deg(y):=\# e(y)). $$

For a set $B \subset Y$ (resp.\ $A \subset X$), we denote
$$e^{-1}(B) : = \{ x \in X : e(x) \subset B \}, \quad (\text{resp. }\, e^{-1}(A) : = \{ y \in Y : e(y) \subset A \}).$$
Note that by definition $e^{-1}(B) \subseteq e(B)$.

\subsection{Double cosets and conjugacy classes}\label{sec:doublconj}
Let $\Gamma\subset \PSL_2(\R)$ be a Fuchsian group of finite covolume with a cusp at infinity of width one. Let 
$$\mathrm{Conj}(\Gamma):=\{\{\gamma\}: \gamma\in \Gamma\},\quad  \{\gamma\}:=\{\sigma \gamma \sigma^{-1}: \sigma\in \Gamma\}\subset \Gamma,$$
be the set of (not necessarily primitive) conjugacy classes of $\Gamma$. Denote by $\Gamma_\infty=\{\begin{psmallmatrix}1& \Z\\0 & 1    
\end{psmallmatrix}\}$ the stabilizer of the cusp $\infty$. Let 
 $$\Gamma_\infty \backslash \Gamma/ \Gamma_\infty:=\{[\gamma]_{\infty}: \gamma\in \Gamma \},\quad [\gamma]_{\infty}:=\Gamma_\infty \gamma \Gamma_\infty=\{\sigma_1\gamma \sigma_2: \sigma_1,\sigma_2\in \Gamma_\infty\} \subset \Gamma,$$
  be the set of double cosets with respect to the pair of cusps $(\infty,\infty)$. We will consider a bipartite graph on the following sets:
\begin{align}
\label{eq:X_N}X_N&:=\{[\gamma]_{\infty} \in \Gamma_\infty \backslash \Gamma/ \Gamma_\infty: 0< |c_\gamma|\leq N\},\\ 
\label{eq:Y_N}Y_N&:=\{\{\gamma\}\in \mathrm{Conj}(\Gamma): N/2\leq |\tr(\gamma)|\leq N\}, \end{align}
for $N\geq 1$. Note that in the notation in eq.\ (\ref{eq:TgammaX}) of the introduction we have $X_N=T_\Gamma(N)$. By respectively, \cite[Eq. (3.6)]{peri18} and the prime geodesic theorem \cite{hejhal2} we have
\begin{equation}\label{eq:sizeXY}
    |X_N|\sim C_1 N^2, \quad |Y_N|\sim C_2 \tfrac{N^2}{\log N},\quad N\rightarrow \infty\end{equation}
for certain constants $C_1,C_2>0$ depending on $\Gamma$. We denote $Y_N^* \subset Y_N$ the subset of {\it{primitive}} conjugacy classes. Note that  $Y^\ast_N=\tilde{C}_\Gamma(N)$, as defined in eq.\ (\ref{eq:tildeCgammaN}). Recall also that the number of non-primitive conjugacy classes is negligible:  
\begin{equation}\label{eq:nonprim}
    |Y_N \setminus Y_N^*|\ll N.
\end{equation}
We define the bipartite graph on $X_N,Y_N$ with an edge if a double coset and conjugacy class intersects non-trivially:
\begin{equation}\label{eq:G_N}G_N:=\{ ([\gamma]_{\infty},\{\gamma'\}) \in X_N\times Y_N:  [\gamma]_{\infty}\cap\{\gamma'\} \neq \emptyset \}.\end{equation} 

Given an element $x=[\gamma]_\infty\in X_N$, we define $$c(x):=|c_\gamma|,$$
to be the absolute lower-left entry of a representative $\gamma$ of the double coset (which we note is independent of the choice of representative). Note that when $\Gamma\subset \PSL_2(\Z)$ then $c(x)=\mathrm{denom}(\gamma \infty)$ is the denominator of the rational number $\gamma\infty\in \Q\cup\{\infty\}$.

For $y=\{ \gamma\}  \in Y_N$, we define $\CC_y\subset \mathcal{X}_\Gamma$ to be the oriented closed geodesic corresponding to the (hyperbolic) conjugacy class $\{ \gamma\} $ and put  
$$\ell(y):=\ell(\CC_y).$$ 

Let $f$ be either a Maa{\ss}  form for $\Gamma$ of weight $k\in 2\Z$ or a completed Eisenstein series with constant term $f_\infty(y)$ at the cusp $\infty$. For $x = [\gamma]_{\infty} \in X_N$, we define the \emph{vertical period of $f$ and $x$} as 
$$L_f(x):= \int_0^{\infty}  (f(\gamma \infty +iy)-f_\infty(y) )\frac{dy}{y}. $$ 
For $y \in Y_N$, denote the associated geodesic period by
$$\mathcal{P}_f(y):= \mathcal{P}_f(\CC_y), $$
where $\mathcal{P}_f(\CC_y)$ is defined as in equation (\ref{eq:geodesicperiods}).
We will use Proposition \ref{prop:main} to transfer information about the distribution of the values $\{ L_f(x) : x \in X_n \}$, which is known by Theorems \ref{thm:normaleffective1}, \ref{thm:normalnoneffective} and \ref{thm:normaleffective2}, to those of $\{ \mathcal{P}_f(y) : y \in Y_n \}$.

\subsection{Controlling the degrees of the graphs}
Let $G_N\subset X_N\times Y_N$ be as in the previous section. As mentioned in the introduction, the proofs of Theorems \ref{main thm nonvanishing}, \ref{thm:main} and \ref{thm:Eis} amount to an upper bound on the degrees $\deg(x)$ for $x\in X_N$ and a lower bound for the degrees $\deg(y)$ for $y\in Y_N$ (on average). 

\subsubsection{Degrees of vertices in $X_N$} This first bound is easily controlled as follows. 
\begin{lemma}\label{lem:c_x}
For $x\in X_N$ it holds that 
$$\deg(x)=\frac{N}{c(x)}+E(x),$$
where $|E(x)|\leq 2$. Here, the degree is taken with respect to the graph (\ref{eq:G_N}). 
\end{lemma}
\begin{proof}
Let $x\in X_N$ and let $$\gamma=\begin{psmallmatrix}
    a_\gamma& b_\gamma\\ c_\gamma& d_\gamma
\end{psmallmatrix}\in \Gamma,$$ 
be such that $x=[\gamma]_{\infty}$ so that $0<c(x)=|c_\gamma|\leq N$. Since multiplication from the left by $\begin{psmallmatrix}
    1& \pm 1\\0&1
\end{psmallmatrix}$ changes the trace by $\pm c_\gamma$, we may arrange it so that $-c(x)/2\leq a_\gamma+d_\gamma<c(x)/2$. We note that for any $k\in \Z\setminus \{0\}$ the conjugacy classes of $\gamma$ and $\begin{psmallmatrix}
    1& k\\0&1
\end{psmallmatrix}\gamma$ are different as the (signed) traces are different. This implies that 
\begin{align} \deg(x)&=\#\{k\in \Z: N/2\leq  |a_\gamma+d_\gamma+kc(x)|\leq N\}.
\end{align}
We have the following elementary estimate:
$$\lfloor \tfrac{N/2}{c(x)}\rfloor\leq \#\{k\in \Z:  N/2\leq a_\gamma+d_\gamma+kc(x)\leq N\}\leq  \lfloor \tfrac{N/2}{c(x)}\rfloor+1,$$
and similarly for traces between $-N$ and $-N/2$. Since $ \tfrac{N/2}{c(x)}-1\leq \lfloor \tfrac{N/2}{c(x)}\rfloor\leq \tfrac{N/2}{c(x)}$ we get the wanted conclussion by adding the two contributions. 
\end{proof}
Note that for typical $x\in X_N$ we have $c(x)\asymp N$ in which case $\deg (x)$ is bounded.
\subsubsection{The degrees of vertices in $Y_N$} We will now give a geometric lower bound for the degrees $\deg(y)$ with $y\in Y_N$. Let $M_\Gamma \geq 1$ and
\begin{equation}
  \label{eq:B}  \mathcal{B}:=\{z\in \H: -1/2< \Re z< 1/2, M_\Gamma < \Im z< 2M_\Gamma \}
\end{equation}
be such that the restriction of the natural projection  $\H \rightarrow \mathcal{X}_\Gamma$ to $\mathcal{B}$ is injective. This exists by the thick-thin decomposition (see e.g. \cite[Chap.\ 4.5]{Thurston97}) since we assumed that the cusp at infinity has width one. We will henceforth identify $\mathcal{B}$ with its image in $\mathcal{X}_\Gamma$. We have the following key geometric lower bound for the degrees of vertices in 
$$Y^\ast_N=\{\{\gamma\}\in \mathrm{Conj}(\Gamma)\text{ primitive}: N/2\leq |\tr \gamma|\leq N\}.$$
\begin{proposition}\label{prop:degy}
Let $N\geq 1$ and let $\mathcal{B}$ be as above. Then there exists a constant $C=C(\mathcal{B})>0$ (independent of $N$) such that for any $y\in Y_N^*$, we have
 $$\deg (y)\geq C\cdot  \ell(\CC_y\cap \mathcal{B}). $$
\end{proposition}
The proof of Proposition \ref{prop:degy} will occupy the rest of this section. We start by recalling the following hyperbolic geometric fact. 
\begin{lemma}
Let $0<M\leq r$ and let $S$ denote the infinite geodesic connecting $-r$ and $r$.
Then $\ell(S\cap \{\Im z\geq M\})=2\log (r+\sqrt{r^2-M^2})-2\log M$. 
\end{lemma}
\begin{proof}
    Applying the matrix  
    $$\begin{pmatrix} (2r)^{-1/2} & -(r/2)^{1/2}\\(2r)^{-1/2} &(r/2)^{1/2}\end{pmatrix}\in \PSL_2(\R),$$   
    takes $S$ to the vertical geodesic from $0$ to $\infty$. The two points in $S\cap \{\Im z= M\}$ are taken to, respectively $z_0=i\tfrac{M^2}{r- \sqrt{r^2-M^2}}$ and $z_1=i\tfrac{M^2}{r+ \sqrt{r^2-M^2}}$. Now the first result follows since M\"{o}bius transformations preserve the geodesic length and 
    \begin{align*}
    \int_{z_0}^{z_1} \frac{|dz|}{\Im z}&= \log\left(r+ \sqrt{r^2-M^2}\right)-\log\left(r- \sqrt{r^2-M^2}\right)\\&=2\log\left(r+ \sqrt{r^2-M^2}\right)-2\log M,  \end{align*}
    as wanted.
\end{proof}
\begin{corollary}
\label{length above Y}
   Let $M>0$ and $r>0$ and let $S$ denote the infinite geodesic connection $-r$ and $r$.
Then it holds that 
$$\ell(S\cap \{M< \Im z< 2M\})\leq 2\log(2+\sqrt{3}).$$   
\end{corollary}
\begin{proof}
If $r< M$ then the intersection is empty. If $M\leq r <2M$ then by the above lemma we get that 
$$\ell(S \cap \{ \Im z>  M\})=2\log (r+\sqrt{r^2-M^2})-2\log M \leq 2\log (2+\sqrt{3}).$$ 
Finally if $r>2M$ then we have by the above that 
$$\ell(S \cap \{ \Im z> M\})=2\log (r+\sqrt{r^2-M^2})-2\log M-2\log (r+\sqrt{r^2-4M^2})+2\log 2M,$$
which has $r$-derivative equal to
$$ \frac{\sqrt{r^2-4M^2}-\sqrt{r^2-M^2}}{\sqrt{r^2-4M^2}\sqrt{r^2-M^2}}.$$
Thus it is a decreasing function for $r\geq 2M$ and at $r=2M$ we recover the same value as in the previous case.
\end{proof}

\begin{proof}[Proof of Proposition \ref{prop:degy}] 

For $\gamma=\begin{psmallmatrix}
    a_\gamma&b_\gamma\\c_\gamma&d_\gamma
\end{psmallmatrix} \in \Gamma$ hyperbolic and primitive with $c_\gamma>0$, denote by $S_{\gamma}$ the axis of $\gamma$, i.e. the infinite geodesic half-circle with end-points in $\R$ given by 
$$\frac{a_\gamma-d_\gamma \pm \sqrt{(a_\gamma+d_\gamma)^2-4}}{2c_\gamma}.$$
For $\gamma,\gamma' \in \Gamma$ primitive hyperbolic and $\sigma\in \Gamma$ one has   $S_{\gamma} = \sigma S_{\gamma'}$ exactly if $ \gamma'= \sigma \gamma^{\pm 1} \sigma^{-1}$.
Also, note that
\begin{equation}
    \label{infinity contribution}
    \ell(\Gamma_{\infty} S_{\gamma} \cap \mathcal{B}) = \ell(S_{\gamma} \cap \Gamma_{\infty} \mathcal{B} ) = \ell(S_\gamma \cap \{M_\Gamma \leq \Im z\leq 2M_\Gamma \}) \ll 1,
\end{equation}
where we used Corollary \ref{length above Y} and $\Gamma_\infty=\{\begin{psmallmatrix}
    1&\Z\\0& 1
\end{psmallmatrix}\}$ denotes the stabilizer of $\infty$. Put in words, this means that the length of all translations of $S_{\gamma}$ intersected with $\mathcal{B}$ is bounded above by a constant.

Now fix a primitive oriented closed geodesic $\mathcal{C}\subset \mathcal{X}_\Gamma$. We consider the geodesic arcs of $\mathcal{C}$ that intersect $\mathcal{B}$ (considered as a subset of $\mathcal{X}_\Gamma$). We have that
$$\mathcal{C} \cap \mathcal{B} = \bigcup_{\gamma \in y(\CC)} (S_{\gamma} \cap \mathcal{B}),$$
where $y(\CC)\in \mathrm{Conj}(\Gamma)$ denotes the primitive conjugacy class associated to the closed geodesic $\mathcal{C}$. Note that if $S_{\gamma} \cap \mathcal{B} \neq \emptyset$, then clearly the radius of the axis $S_\gamma$ of $\gamma$ satisfies $r_{\gamma} \geq M_\Gamma \geq 1$. By the explicit formula:
$$r_{\gamma} = \frac{\sqrt{\tr(\gamma)^2-4}}{2 c_{\gamma}},$$
we conclude that $c_{\gamma} \leq |\tr(\gamma)|$. This shows that if $\{\gamma\}\in Y_N^*$ and $S_\gamma \cap \mathcal{B}\neq \emptyset$ then we have $[\gamma]_\infty \in X_N$, and thus $e=([\gamma]_{\infty} , \{ \gamma\})$ is an edge of $G_N$. Now if $\gamma'\in \Gamma$ satisfies $\gamma' \in \{\gamma\}$ and $\gamma' \in [\gamma]_\infty$ then we have
$$ \tr(\gamma)=\tr(\gamma')\quad \text{and}\quad \gamma=\begin{psmallmatrix}
    1& m\\ 0 &1
\end{psmallmatrix}\gamma' \begin{psmallmatrix}
    1& -n\\ 0 &1
\end{psmallmatrix}, $$
for some $m,n\in \Z$. Observe that we have $\tr(\begin{psmallmatrix}
    1& m\\ 0 &1
\end{psmallmatrix}\gamma'\begin{psmallmatrix}
    1& -n\\ 0 &1
\end{psmallmatrix})= \tr(\gamma')+(m-n)c_{\gamma'}$ where $c_{\gamma'}\neq 0$ denotes the lower-left entry of $\gamma'$  (which is non-zero since $\gamma'$ is hyperbolic). This means that $m=n$ and thus the elements $\gamma'\in y(\CC)$ contributing to the edge $e$ are exactly the conjugates of $\gamma$ by elements of $\Gamma_{\infty}$. But from \eqref{infinity contribution} and the preceding observation, we know that the total contribution of these conjugates to the total length of $\mathcal{C} \cap \mathcal{B}$ is $O(1)$. This shows indeed that $\ell(\mathcal{C} \cap \mathcal{B}) \ll \deg (\{\gamma \})$ as wanted.
\end{proof}

\subsubsection{Sparse equidistribution of closed geodesics}\label{sec:sparse}
In this section we prove a sparse equidistribution result for closed geodesic using the ergodicity of the geodesic flow, as well as a strenghtening for the modular group, using a result of Aka--Einsiedler \cite[Thm.\ 2]{akaeins}. Combined with Proposition \ref{prop:degy} this will allow us to control the degrees of the vertices in $Y_N$.  

Let $\phi:\Gamma\backslash\PSL_2(\R)\rightarrow \C$ be a continuous map. For $T\geq 1$ we denote by 
\begin{equation}\label{eq:flow}\phi_T:\Gamma\backslash\PSL_2(\R)\rightarrow \C,\quad g\mapsto\frac{1}{T}\int_0^T \phi(ga_t)dt,\quad a_t=\begin{psmallmatrix}e^{t/2}&0\\0 &e^{-t/2}\end{psmallmatrix},\end{equation}
the average of $\phi$ under the geodesic flow. Recall that the geodesic flow satisfies the $L^1$-ergodic theorem (see \cite[Prop. 4]{akaeins} for a stronger statement), meaning that for $\phi:\Gamma\backslash\PSL_2(\R)\rightarrow \C$ continuous and bounded it holds that:
\begin{equation}\label{eq:ergodic}\phi_T\xrightarrow{L^1} \langle \phi,1\rangle, \quad T\rightarrow \infty,
\end{equation}
in $L^1$-sense with respect to the Haar measure on $\Gamma\backslash \PSL_2(\R)$.
\begin{theorem}
\label{thm:sparse}
Let $\Gamma$ be a Fuchsian group of finite covolume. Fix $\epsilon>0$. For each $X\geq 1$ let 
$$I_X\subset C_\Gamma(X)=\{\CC\subset \mathcal{X}_\Gamma \text{ \rm  primitive closed geodesic}: \ell(\CC)\leq X\},$$
be a subcollection of closed geodesics such that 
$$ \ell(I_X)\geq  \epsilon\cdot \ell(C_\Gamma(X)). $$
Then the collections of closed geodesics $I_X$ equidistribute in the unit tangent bundle of $\mathcal{X}_\Gamma$ as $X\rightarrow \infty$. Here $\ell(I_X)$ and $\ell(C_\Gamma(X))$ are defined as in eq.\ (\ref{eq:collection}).
\end{theorem}
\begin{proof} 
By the equidistribution of closed geodesics $C_\Gamma(X)$ in the unit tangent bundle as in e.g. \cite[Cor.\ 6.5]{Zelditch89}, we have that for any continuous and bounded function $\phi:\Gamma\backslash \PSL_2(\R)\rightarrow\C$, it holds that
\begin{equation}
    \label{eq:ZelditchEQUI}\frac{1}{\ell(C_\Gamma(X))} \sum_{\CC\in C_\Gamma(X)} \mu_\CC(\phi)\rightarrow \langle \phi,1\rangle,\quad X\rightarrow\infty. 
\end{equation} 
Now let $\phi:\Gamma\backslash \PSL_2(\R)\rightarrow \C$ be continuous and bounded. Then we have to show that 
$$ \frac{1}{\ell(I_X)} \sum_{\CC\in I_X} \mu_\CC(\phi)\xrightarrow{?} \langle \phi,1\rangle. $$
Since the measure on the left hand side is a probability measure, we may reduce to the case $\langle\phi,1\rangle=0$. By the invariance of $\mu_\CC$ under the geodesic flow we have for any $T\geq 1$:
\begin{align}
 \nonumber \left|\frac{1}{\ell(I_X)} \sum_{\CC\in I_X} \mu_\CC(\phi)\right|= \left| \frac{1}{\ell(I_X)} \sum_{\CC\in I_X} \mu_\CC(\phi_T)\right| &\leq  \frac{\ell(C_\Gamma(X))}{\ell(I_X)} \frac{1}{\ell(C_\Gamma(X))} \sum_{\CC\in I_X} \mu_\CC(|\phi_T|)\\
 \label{eq:smallmeasure}&\leq \frac{\epsilon^{-1}}{\ell(C_\Gamma(X))} \sum_{\CC\in C_\Gamma(X)} \mu_\CC(|\phi_T|).
\end{align}
Note that by eq.\ (\ref{eq:ZelditchEQUI})  the right hand side of the above converges to $\langle |\phi_T|,1\rangle $ as $X\rightarrow \infty$ (for any fixed $T$). Now by the $L^1$-ergodic theorem (\ref{eq:ergodic}), the quantity $\langle |\phi_T|,1\rangle=\int_{\Gamma\backslash\PSL_2(\R)}|\phi_T| $ converges to $0$ as $T\rightarrow \infty$ (using here that $\langle \phi,1\rangle=0$). It follows that indeed the quantity (\ref{eq:smallmeasure}) can be made arbitrarily small for $X$ and $T$ sufficiently large.  
\end{proof}
In the case of the modular group we have the following  improvement by a result of Aka--Einsiedler \cite[Thm.\ 2]{akaeins}.
\begin{theorem}
\label{thm:AkaEin}
Let $\mathcal{X}_0(1):=\PSL_2(\Z)\backslash \H$ be the modular surface. Let $\psi : \R_{>0} \to \R_{>0}$ a function such that $\lim_{X \to \infty} \psi(X) = \infty$. For each $X\geq 1$ let 
$$I_X\subset C(X)=\{\CC\subset \mathcal{X}_0(1)\text{ \rm  primitive closed geodesic}: \ell(\CC)\leq  X\} $$
be a subcollection of closed geodesics such that 
$$ \ell(I_X)\geq  \psi(X) \frac{\ell(C(X))}{X }. $$
Then the collections of closed geodesics $I_X$ equidistribute in the unit tangent bundle of the modular  surface $\mathcal{X}_0(1)$ as $X\rightarrow \infty$.  Here $\ell(I_X)$ and $\ell(C(X))$ are defined as in eq.\ (\ref{eq:collection}).
\end{theorem}
\begin{proof}
We may reduce to the case where $\psi$ is increasing. As explained in \cite{sarnak82} the primitive oriented closed geodesics on the modular curve $\PSL_2(\Z)\backslash \H$ are parameterized by elements of narrow class groups $\Cl^+_D$ of quadratic orders of (ring) discriminant $D>0$. The length of a geodesic of discriminant $D$ equals $2\log \epsilon_D$, where $\epsilon_D=u+\sqrt{D}v$ is the fundamental positive unit of the quadratic order $\mathcal{O}_D$ of discriminant $D$ (i.e. $u^2-Dv^2=1$ is the fundamental solution with $u,v\in \tfrac{1}{2}\Z_{\geq 1}$). We denote 
    $$\mathcal{D}(X):= \{ D>0 \text{ ring discriminant } : 2 \log \epsilon_D \leq X\}.$$
    Then from \cite{sarnak82}, we know that  for any $\epsilon>0$,
    \begin{equation}
        \label{eq:D(X)}|\mathcal{D}(X)| = \tfrac{35}{16} e^{X/2} + O_\epsilon\lr{e^{(1/3+\epsilon)X}}.
    \end{equation}
    Let $\mathcal{C}_D$ denote the set of primitive oriented closed geodesics on $\mathcal{X}_0(1)$ of discriminant $D$ and put $I_D:=\CC_D\cap I_X$. 
   Put
    $$\mathcal{A}(X):= \left \{ D \in \mathcal{D}(X) \ : \ \ell(I_D) \geq \psi(X)^{1/2} \frac{\ell(\CC_D)}{\log \ell(\CC_D)}\right \}.$$
    Note that $X\geq \log (D/4)$ since $\epsilon_D\geq \sqrt{D}/2$, which implies that $\psi(X)\geq \psi(\log (D/4))$.  From the class number formula, we have the (effective) upper bound $\ell(\mathcal{C}_D) = |\mathcal{C}_D| \log \epsilon_D^2 \ll_\epsilon D^{1/2+\epsilon}$ for $\epsilon>0$. This implies that $\log \ell(\CC_D)\leq \log D$ for $D$ sufficiently large.  Thus, by the definition of $\mathcal{A}(X)$ and since $\psi(\log (D/4))\rightarrow \infty$ as $D\rightarrow \infty$, we conclude  from \cite[Theorem 2]{akaeins} that for $D \in \mathcal{A}(X)$, the geodesics from $I_D$ equidistribute as $D \to \infty$. This implies that the union of geodesics $\cup_{D \in \mathcal{A}(X)} I_D$ equidistributes as $X\rightarrow \infty$ by a standard averaging argument. Therefore the conclusion follows if we can show
    $$\sum_{D \in \mathcal{D}(X) \setminus \mathcal{A}(X)} \ell(I_D) \stackrel{?}{=} o \lr{\ell(I_X)}.$$
    To achieve this, we split in terms of the size of $\ell(\CC_D)$:
     \begin{align*}
        \sum_{D \in \mathcal{D}(X)\setminus \mathcal{A}(X)} \ell(I_D)
        & = \sum_{\substack{D \in \mathcal{D}(X)\setminus \mathcal{A}(X)\\ \ell(\CC_D)\leq e^{X/3}}} \ell(I_D) +\sum_{\substack{D \in \mathcal{D}(X)\setminus \mathcal{A}(X)\\ \ell(\CC_D)> e^{X/3}}} \ell(I_D)\\
        &\leq |\mathcal{D}(X)|e^{X/3}+\psi(X)^{1/2} \sum_{D\in \mathcal{D}(X)}\frac{\ell(\CC_D)}{X/3} \\
        &\ll e^{5X/6}+\psi(X)^{1/2}\frac{\ell(C(X))}{X},
        \end{align*}
        where we used the bound (\ref{eq:D(X)}) in the last inequality. Recall that by the prime geodesic theorem we have $\ell(C(X))\sim e^X$. By the assumption on $\ell(I_X)$ and since $\psi(X)^{1/2}\rightarrow \infty$, we conclude that the above is indeed $o(\ell(I_X))$ as wanted.
\end{proof}
\begin{remark}
It is likely that one can obtain a proof of the sparse equidistribution in Theorem \ref{thm:AkaEin} for a general Fuchsian group of finite covolume. By the method of Aka--Einsiedler \cite{akaeins}, this reduces to obtaining a version of   Zelditch's effective equidistribution theorem \cite[Thm.\ 6.1]{Zelditch89} with polynomial dependence on the spectral parameter. This should follow by a detailed analysis of the hypergeometric functions appearing in Zelditch's trace formula, see \cite[p.\ 85]{zel92}, as well as invoking a version of Sarnak's bound for triple periods \cite{sarnak94} with a polynomial dependence of the spectral datum of the ``fixed'' Maa{\ss} forms (as was carried out in a special case in \cite[Appendix A]{HuangLester23}). We have not pursued this.    
\end{remark}
\section{Non-vanishing of geodesic periods}
\label{sec:nonvanishgeo}



In this section we will prove non-vanishing results for geodesic periods using the results in the previous sections. We will prove two versions: one where we obtain strong quantitative bounds on the number of ``very small'' geodesic periods, and on the other hand, we show that $100\%$ of geodesic periods are not ``too small''. 

We will keep the notation from the previous sections. The following is a slight quantitative strengthening of Theorem \ref{main thm nonvanishing}. 
\begin{theorem}
\label{small length bound}
    Let $f$ be a Hecke--Maa{\ss} cusp form of weight $0$ for the modular group $\Gamma=\PSL_2(\Z)$ and let $\delta>0$. Then we have that 
    $$\left|\left\{ y \in Y_N^* : \left| \mathcal{P}_f(y)\right|\leq (\log N)^{1/2-\delta}   \right\}\right| \ll \frac{N^2}{(\log N)^{1+\min (\delta,1/4)}},$$
    as $N\rightarrow \infty$. Here $Y_N^\ast$ denotes primitive conjugacy classes with trace between $N/2$ and $N$ as defined in Section \ref{sec:doublconj}.
\end{theorem}
Before giving the proof, we note that $Y_N^*$ corresponds to conjugacy classes with trace between $N/2$ and $N$ whereas $C(2\log N)$ as in equation (\ref{eq:CX}) corresponds to trace between $2$ and $N$ (i.e. putting $X=2\log N$). By doing a dyadic decomposition one can easily conclude  Theorem \ref{main thm nonvanishing} from the above theorem. We will skip the details.
\begin{proof}
Put $\alpha=1+\min (\delta,1/4)$ and  
$$A_N: = \left\{ y \in Y_N^* : \left| \mathcal{P}_f(y)\right|\leq (\log N)^{1/2-\delta}   \right\}.$$
Clearly we may suppose that 
    $|A_N| \geq c' \frac{N^2}{(\log N)^{\alpha}}$, where $c'$ is a sufficiently large constant. This implies that
\begin{align}\label{eq:equidAnBn}
    \sum_{y \in A_N} \ell(\CC_y) \gg \log N\frac{N^2}{(\log N)^{\alpha}}\gg \frac{N^2}{(\log N)^{1/4}}.
\end{align}

Let $B_N :=e(A_N)  \subset X_N$ be the set of neighbours of  $A_N$. Recall that $\tr(y)\geq N/2$ for $y\in Y_N$. From Proposition \ref{prop:main}, we see that for $x \in B_N$,
$$|L_f(x)| \leq \left|\mathcal{P}_f(y)\right|+O\left(1\right)\ll (\log N)^{1/2-\delta}.$$
In particular, there exists some constant $c''>0$ such that $x\in B_N$ implies that $|L_f(x)/(\log N)^{1/2}|\leq c'' (\log N)^{-(\alpha-1)} $. From the normal distribution of the vertical periods as in Theorem \ref{thm:normaleffective2}, we have that
\begin{align}\label{eq:normaldistbound}
    \frac{|B_N|}{|X_N|} \leq \int_{|x|\leq c''(\log N)^{-(\alpha-1)}} \frac{e^{-x^2/2}}{\sqrt{2 \pi}} dx + O \lr{\frac{1}{\sqrt{\log N}}} \ll (\log N)^{-2(\alpha-1)}, 
\end{align}
using that $2(\alpha-1)=2\min (1/4,\delta)\leq 1/2$.

We now proceed to count the total number of edges out of both $A_N$ and $B_N$ in two different ways. On the one hand, we have
\begin{align*}
    \sum_{y \in A_N} \deg y &\gg \sum_{y \in A_N} l(\mathcal{C}_y \cap \mathcal{B})\gg \sum_{y \in A_N} l(\mathcal{C}_y) \gg  |A_N| \log  N 
\end{align*}
where  in the first inequality we applied Proposition \ref{prop:degy} (recall the definition of $\mathcal{B}$ in eq.\ (\ref{eq:B})), in the second we applied  the equidistribution result Theorem \ref{thm:AkaEin} (with $X=2\log N$) to the closed geodesics in $A_N$ which applies by the lower bound (\ref{eq:equidAnBn}). 

On the other hand, we can upper bound the degrees in $B_N$ as follows for any $\beta>0$:

\begin{align*}
    \sum_{x \in B_N} \deg x = \sum_{\substack{x \in B_N: \\ c_x \leq \frac{N}{(\log N)^{\beta}}}} \deg x  + \sum_{\substack{x \in B_N: \\ c_x > \frac{N}{(\log N)^{\beta}}}} \deg x &\leq \sum_{\substack{x \in X_N: \\ c_x \leq \frac{N}{(\log N)^{\beta}}}} \deg x + 2 \sum_{x \in B_N} (\log N)^{\beta} \\
    &\ll \frac{N^2}{(\log N)^{ \beta}} + |B_N| (\log N)^{\beta} \\
    &\ll \frac{N^2}{(\log N)^{\beta}} + \frac{N^2}{(\log N)^{2-2\alpha - \beta}},
\end{align*}
where in the second to last inequality we have applied equation (\ref{eq:sizeXY}) and  Lemma \ref{lem:c_x}, and the last inequality follows from the upper bound (\ref{eq:normaldistbound}) on the size of  $B_N$.
Choosing $\beta=\alpha-1$, we obtain an upper bound of $N^2/(\log N)^{\alpha-1}$.
Notice that since $B_N\subset X_N$ consists of all the neighbours of $A_N$, we have the trivial inequality
$$\sum_{x \in B_N} \deg x\geq \sum_{y \in A_N} \deg y$$
Combining all of the above we arrive at 
$$  |A_N|\ll \frac{1}{\log N}\sum_{y\in A_N} \deg y \leq \frac{1}{\log N}\sum_{x\in B_N} \deg x\ll \frac{N^2}{(\log N)^{\alpha-1+1}}, $$
as wanted.
\end{proof}

We will now proceed to prove Theorems \ref{thm:main} and \ref{thm:Eis} from the introduction. The proofs are very similar to the preceding one with the only difference being that for periods of Eisenstein series we will only have to control the contribution from $x\in X_N$ for which $c(x)$ is small.
\begin{proof}[Proof of Theorems \ref{thm:main} and \ref{thm:Eis}]
As above we may reduce to a dyadic interval and consider conjugacy classes in $Y_N^\ast$. We may assume that $h(N)\leq \log\log \log N$. Put 
\begin{align*}
A_N^- = &\left\{ y \in Y_N^* : \left|\mathcal{P}_f(y)\right| \leq (\log N)^{1/2} (\log \log N)^{\delta_f}/h(N) \right\}, \\
A_N^+ = &\left\{ y \in Y_N^* : \left|\mathcal{P}_f(y)\right|\geq (\log N)^{1/2} (\log \log N)^{\delta_f} h(N)  \right\}, \\
A_N = & A_N^- \cup A_N^+, \end{align*}
where $\delta_f=3/2$ if $f=E^\ast$ is the completed Eisenstein series of level 1 and weight 0 and $\delta_f=0$ otherwise.
We want to show that $|A_N|=o(N^2/\log N)$ as $N\rightarrow \infty$. Let $\epsilon>0$ and assume for contradiction that $|A_N|\geq \epsilon N^2/\log N$ for infinitely many $N$. 
Let \begin{equation}
  B_N^\pm=e(A_N^\pm),\quad   D_N=\{x \in X_N: c(x)\leq N/(\log N)^{1/2} \}.
\end{equation} 
Then by Proposition \ref{prop:main} we have that for $x\in B_N^- \setminus D_N$:
\begin{align*}
    |L_f(x)| &\ll (\log N)^{1/2}  (\log \log N)^{\delta_f} /h(N)+ (\log N)^{1/4 + 1/100} + (\log\log N)^{1/2+1/100}\\
    &\ll (\log N)^{1/2}  (\log \log N)^{\delta_f} /h(N),  
\end{align*}
and similarly for $x \in B_N^+ \setminus D_N$. Note that since $h(N)\rightarrow \infty$, we have
\begin{align}
    \mathbb{P}(|\mathcal{N}_\R(\sigma,\mu)|\leq h(N)^{-1})+\mathbb{P}(|\mathcal{N}_\R(\sigma,\mu)|\geq h(N))&\rightarrow 0,\quad N\rightarrow \infty,
\end{align}
where $\mathcal{N}_\R(\sigma,\mu)$ denotes a  real valued normally distributed random variable with mean $\mu$ and variance $\sigma$. The same holds for a complex valued normally distributed random variable $\mathcal{N}_\C(\sigma,\mu)$. Thus, with $B_N=B_N^+ \cup B_N^-$, we conclude by the normal distribution results in Theorems \ref{thm:normaleffective1}, \ref{thm:normalnoneffective}, \ref{thm:normaleffective2} and \ref{thm:normalnoneffective2} that 
$$ \frac{|B_N \setminus D_N|}{|X_N|}\rightarrow 0,\quad N\rightarrow \infty.$$
Let  $\psi: \R_{>0}\rightarrow \R_{>0}$ be an increasing (and sufficiently slowly growing) function such that $\psi(N)\rightarrow \infty$ and $|B_N \setminus D_N| \psi(N)=o(N^2)$ as $N\rightarrow \infty$ . We proceed as above and count the number of edges between $B_N \setminus D_N$ and $A_N$ in two different ways. First of all the total number of such edges is clearly upper bounded by 
\begin{align}
   \nonumber \sum_{x\in B_N \setminus D_N}\deg x &= \sum_{\substack{x \in B_N \setminus D_N: \\ c_x \leq \frac{N}{\psi(N)}}} \deg x  + \sum_{\substack{x \in B_N \setminus D_N: \\ c_x > \frac{N}{\psi(N)}}} \deg x\\
  \label{eq:degreeBN}  &\ll \frac{N^2}{\psi(N)}+ |B_N \setminus D_N|\psi(N)  = o(N^2).
\end{align}
On the other hand, by the assumptions on the size of $A_N$ Theorem \ref{thm:sparse} applies and we conclude  by equidistribution and the lower bound Proposition \ref{prop:degy}:
\begin{equation}\label{eq:degreeAN}\sum_{y\in A_N}\deg y \gg \sum_{y\in A_N}\ell(\CC_y\cap \mathcal{B})\gg  |A_N| \log N \gg_\epsilon   N^2.\end{equation}
 Recall that by definition $B_N$ is the set of neighbors of $A_N$, so the number of edges between $A_N$ and $B_N \setminus D_N$ is lower bounded by
 \begin{align*}
\sum_{y\in A_N}\deg y - \sum_{x \in D_N} \deg x \geq \sum_{y\in A_N}\deg y - 2\frac{N^2}{(\log N)^{1/2}} \gg_\epsilon N^2,
 \end{align*}
using the bound (\ref{eq:degreeAN}). This contradicts the bound in (\ref{eq:degreeBN}).
\end{proof}
\begin{remark}
Using a similar argument as in the proof of Theorem \ref{small length bound}, one can obtain upper bounds for the set of large geodesic periods. More precisely, for any $\delta>0$, we have
$$\left |  \left\{  y \in Y_N^* : \left| \mathcal{P}_f(y) \right|\geq (\log N)^{1/2+\delta}   \right\} \right | \ll_{\delta} \frac{N^2}{(\log N)^2}.$$
Note that this is stronger than the bound obtained in Theorem \ref{small length bound}. The additional gain comes from the fact that we have a strong upper bound for the set of large vertical line periods. That is, we have for any $A>0$,
$$\left | \left \{ x \in X_N: |L_f(x)| \geq (\log N)^{1/2+\delta} \right \}\right | \ll_{\delta,A} \frac{N^2}{(\log N)^A}.$$
Such an estimate follows easily from the $2k$-th moment \cite[Thm.\ 1.6]{No21}, \cite[Thm.\ 7.5]{PeRi2}, \cite[Prop.\ 7.9]{DrNo}: for each $k\geq 0$ there exists $d_k>0$ such that
$$\frac{1}{|X_N|}\sum_{x \in X_N} |L_f(x)|^{2k} \sim d_k (\log N)^k, \quad N\rightarrow \infty.$$
\end{remark}
\section{Normal distribution with weighted measures}\label{sec:normal}
In this section we prove Theorem \ref{thm:CLTintro}. We begin by considering a general setup with measures on a bipartite graph and give conditions for when one can ``lift'' the distribution from one component to the other. More precisely, we consider the push-forward with respect to the bipartite of a measure on one component to the other. We will see that, by Proposition \ref{prop:main}, the relevant conditions are satisfied for the bipartite graph defined in terms of double cosets and conjugacy classes. Combining with the normal distribution for vertical periods as in Theorems \ref{thm:normaleffective1}, \ref{thm:normalnoneffective} and \ref{thm:normaleffective2}, we will conclude the proof of Theorem \ref{thm:CLTintro}.   
\subsection{Measures on bipartite graphs}
\label{sec:graphs}
Let $(X,\mu),(Y,\nu)$ be finite discrete probability (measure) spaces, i.e. $X,Y$ are finite sets equipped with the $\sigma$-algebra consisting of all subsets and $\mu$ (resp.\ $\nu$) are probability measures on $X$ (resp.\ $Y$). As above let $G\subset X\times Y$ be a bipartite graph on $X,Y$. We will denote by $G(\mu)$ the \emph{$G$-transform of $\mu$}, which is the probability measure on $Y$ given by 
\begin{equation}
\label{eq:Gtransform}G(\mu)(\{y\}):=\sum_{x\in e(y)}\frac{\mu(x)}{\deg(x)},\quad y\in Y,\end{equation}
and similarly we define the measure $G(\nu)$ on $X$. Note that for $x\in e(y)$, we have $y\in e(x)$, and thus $\deg (x)\geq 1$. 
\begin{lemma}\label{lem:cute}
For any subset $B\subset Y$ we have that
$$\mu(e^{-1}(B)) \leq G(\mu)(B)\leq \mu(e(B)).$$
\end{lemma}
\begin{proof}
We have by definition of $G(\mu)$:
\begin{align*}G(\mu)(B)=\sum_{y\in B}\sum_{x\in e(y)}\frac{\mu(x)}{\deg(x)}&= \sum_{x\in e(B)}\frac{\mu(x)}{\deg(x)} \cdot \#(e(x)\cap B)\\
&\leq \sum_{x\in e(B)}\frac{\mu(x)}{\deg(x)} \cdot \deg(x)= \mu(e(B)). \end{align*}
Similarly,
$$G(\mu)(B)= \sum_{x\in e(B)}\frac{\mu(x)}{\deg(x)} \cdot \#(e(x)\cap B) \geq \sum_{x\in e^{-1}(B)}\frac{\mu(x)}{\deg(x)}  \cdot \deg(x)= \mu(e^{-1}(B)). $$
\end{proof}
\subsection{Lifting the distribution}\label{sec:lifting}

Let $(G_n)_{n \geq 1}$ be a sequence of bipartite graphs on $(X_n, Y_n)_{n \geq 1}$. For each $n\geq 1$, let $X_n$ be equipped with a probability measure $\mu_n$. In addition, suppose we have $f_n: X_n \to \R$ such that the push-forward $(f_n)^*(\mu_n)$ converges as $n\rightarrow \infty$ to a distribution on $\R$ with density function $F:\R\rightarrow \R_{\geq 0}$. In other words, for $a<b$ we have
$$\lim_{n \to \infty} \mu_n \left(\{ x \in X_n : f_n(x) \in [a,b] \}\right) = \int_{a}^b F(x) dx . $$

Moreover, we say that we have \emph{explicit rate of convergence} if there exists $h: \R_{>0} \to \R_{>0}$ such that $\lim_{x \to \infty} h(x) = 0$
and 
\begin{equation}
\label{rate of conv}
\sup_{a<b} \left |\mu_n \left(\{ x \in X_n : f_n(x) \in [a,b] \}\right) -  \int_{a}^b F(x) dx \right | = O(h(n)). 
\end{equation}

We want to show that if we can define $g_n:Y_n \to \R$ such that $f_n(x)$ and $g_n(y)$ are ``close'' whenever $x$ and $y$ are connected by an edge in the bipartite graph $G_n$, then we can obtain information about the distribution of $(g_n)^*(G_n(\mu_n))$. In other words, using the structure of the graph $G_n$, we can ``lift'' the limit value distribution of $f_n$ on $X_n$ to a limit value distribution of $g_n$ on $Y_n$.

\begin{lemma}
\label{real graph}
    Let $(G_n)_{n\geq 1}$ be a sequence of bipartite graphs on $(X_n, Y_n)_{n\geq 1}$. For each $n\geq 1$, let $\mu_n$ be a probability measure on $X_n$ and let $f_n:X_n \to \R$ be as above such that \eqref{rate of conv} holds for a continuous density function $F$ and some  $h:\R_{>0} \to \R_{>0}$ such that $\lim_{x \to \infty} h(x) = 0$. Moreover, suppose there exists $g_n: Y_n \to \R$ such that whenever there is an edge between $x \in X_n$ and $y \in Y_n$ in $G_n$, we have $g_n(y)=f_n(x) + O(E(n))$, where $E: \R_{>0} \to \R_{>0}$ satisfies $\lim_{x \to \infty} E(x) = 0$. Then we have 
    $$ \sup_{a<b} \left|G_n(\mu_n) (\{ y \in Y_n : g_n(x) \in [a,b] \}) -  \int_{a}^b F(x) dx \right | = O_F(h(n) + E(n)),$$
    where the implied constant is independent of $a<b$.
\end{lemma}
\begin{proof}
    Let $a<b$ and put $$B_n:=\{ y \in Y_n : g_n(y) \in [a,b] \}. $$
    Using Lemma \ref{lem:cute}, we have that
    $$ \mu_n(e^{-1}(B_n)) \leq G_n(\mu)(B_n)\leq \mu_n(e(B_n)). $$
    Moreover, we have that
    $$e(B_n) \subseteq \{ x \in X_n : f_n(x) \in [a-E(n), b+E(n)] \}$$ and hence 
    $$G_n(\mu_n)(B_n) \leq \int_{a-E(n)}^{b+E(n)} F(x) dx + O(h(n)) = \int_a^b F(x) dx + O_F(h(n) + E(n)).$$
    Here we used that $F$ is a continuous density function, therefore indeed we have 
    $$\int_b^{b+E(n)} F(x) dx= O_F(E(n)).$$

    Similarly, we have that 
     $$e^{-1}(B_n) \supseteq \{ x \in X_n : f_n(x) \in [a+E(n), b-E(n)] \}$$
     and the conclusion follows.
\end{proof}


It is natural to ask a similar question about complex-valued functions $f_n: X_n \to \C$ and $g_n: Y_n \to \C$. In this case, we can obtain a similar result about the distributions of $(f_n)^*(\mu_n)$ and $(g_n)^*(G(\mu_n))$, but without uniformity: Suppose we know that there exists a density function $F: \C \to \R_{\geq 0}$ such that for all rectangles 
$$\mathcal{R}=\mathcal{R}(w;R_1,R_2):=\{ z\in \C: -R_1\leq \Re (z-w)\leq R_1, -R_2\leq \Im (z-w)\leq R_2  \},$$ we have
\begin{equation}
    \label{complex ball integral}
    \lim_{n \to \infty} \mu_n \left \{ x \in X_n : f_n(x) \in \mathcal{R}\right \} = \int_\mathcal{R} F(z) dz.  
\end{equation}
\begin{lemma}
\label{complex graph}
Let $f_n:X_n \to \C$ such that $(f_n)^*(\mu_n)$ obeys asymptotically a distribution given by continuous density function $F$, that is \eqref{complex ball integral} holds. Let $g_n: Y_n \to \C$ be a function such that for all $(x,y)\in G_n$ we have $|f_n(x)-g_n(y)| \leq E(n)$ for some $E:\R_{\geq 0} \to \R_{\geq 0}$ with $\lim_{n \to \infty} E(n)=0$. Then for any rectangle $\mathcal{R}=\mathcal{R}(w;R_1,R_2) \subset \C$, we have
\begin{equation}
    \lim_{n\to \infty} G_n(\mu_n) \left(\left \{ y \in X_n : g_n(y) \in \mathcal{R}\right \}\right) = \int_{\mathcal{R}} F(z) dz. 
\end{equation}
\end{lemma}

\begin{proof}
We proceed as in the proof of Lemma \ref{real graph}. Let 
$$A_n: = \left \{ y \in Y_n : g_n(y) \in \mathcal{R}\right \}.$$
Then 
$$e(A_n) \subseteq \{ x \in X_n : f_n(x) \in \mathcal{R}(w; R_1+ E(n),R_2+E(n))\},$$ 
by the assumption on $f_n$ and $g_n$. Similarly,  
$$e(A_n) \supseteq \{ x \in X_n : f_n(x) \in \mathcal{R}(w; R_1- E(n),R_2-E(n))\}.$$
Using Lemma \ref{lem:cute}, we have 
\begin{align*}
    &\mu_n (\{ f_n^{-1}\lr{\mathcal{R}(w; R_1-E(n),R_2-E(n)} \}) \\ &\leq G_n(\mu_n)(A_n)
    \leq \mu_n (\{ f_n^{-1}\lr{\mathcal{R}(w; R_1+ E(n),R_2+E(n)} \}).
\end{align*}
Conclusion follows by letting $n \to \infty$ since the density function $F$ is assumed continuous.
\end{proof}

\begin{remark}
    Using this approach, in the complex case we cannot obtain uniformity in the error term, for all rectangles $\mathcal{R} \subset \C$. This boils down to the fact that the error term is given by 
    $$\int_{\mathcal{R}(w;R_1 + E(n),R_2+E(n))} F(z) dz - \int_{\mathcal{R}(w;R_1,R_2)} F(z)$$
    and the region $\mathcal{R}(w;R_1 + E(n),R_2+E(n)) \setminus \mathcal{R}(w;R_1,R_2)$ has size $\approx 2(R_1+R_2) \cdot E(n)$ (as opposed to size $E(n)$ in the one dimensional real case).
\end{remark}

\subsection{Normal distribution and a proof of Theorem \ref{thm:CLTintro}}
In this section we will ``lift'' the normal distribution of the vertical periods $\{L_f(x):x\in X_N\}$ to the geodesic periods using the result of the preceding section. Recall the definition of the bipartite graph $G_N\subset X_N\times Y_N$ as in eq.\ (\ref{eq:G_N}) defined in terms of double cosets (\ref{eq:X_N}) and conjugacy classes (\ref{eq:Y_N}) of $\Gamma$. We equip $X_N$ and $Y_N$ with the probability measures $\mu_N$ and $\nu_N$ proportional to the counting measures: 
$$\mu_N(\{x\})= \frac{1}{|X_N|}, \quad \nu_N(\{y\})=\frac{1}{|Y_N|}.$$

To simplify notation, we denote by $\mu_N': Y_N \to \R_{\geq 0}$ the $G_N$-transform of  $\mu_N$;
$$\mu_{N}'(\{y\}) := G_N(\mu_N)(\{y\})= \frac{1}{|X_N|} \sum_{x \in e(y)} \frac{1}{\deg (x)}.$$

Firstly, we will prove an effective normal distribution for the real part of geodesic periods with respect to $\mu_N'$.

\begin{corollary}\label{cor:CLT}
    Let $f$ be either a real valued Hecke--Maa{\ss} cusp form of weight $0$ and level $1$, associated to a holomorphic cusp form of level $1$ or to a holomorphic cusp form of weight $2$ for a  Fuchsian group $\Gamma$ of finite covolume with a cusp. Then for any $a<b$, 
    $$ \mu_{N}' \lr{\left \{ y\in Y_N : \Re \lr{\frac{\mathcal{P}_f(y)}{c_f \sqrt{\log N}}} \in [a,b] \right \}} = \frac{1}{\sqrt{2 \pi}} \int_{a}^b e^{-\frac{x^2}{2}} dx + O_f \lr{\frac{1}{\sqrt{\log N}}},$$
    where $c_f>0$ is as in (\ref{eq:variance}). 
\end{corollary}

\begin{proof}
We will apply Lemma \ref{real graph} to the graph $G_N$ and the functions 
    $$f_N(x) = \Re\lr{\frac{L_f(x)}{c_f \sqrt{\log N}}}, \quad g_N(y) = (-1)^{k/2+1} \Re\lr{\frac{\mathcal{P}_f(y)}{c_f \sqrt{\log N}}}. $$
    Note that for $x\in X_N$ and $y\in Y_N$ we have $c(x)\leq N$ and $\tr(y)\geq N/2$. Thus it follows by Proposition \ref{prop:main} that for any edge $(x,y)$ in  $G_N$ we have $f_N(x)=g_N(y)+O((\log N)^{-1/2}) $. Theorems \ref{thm:normaleffective1} and \ref{thm:normaleffective2} (with $\mathcal{R}=[a,b]\times (-\infty,\infty)$) imply that the effective estimate (\ref{rate of conv}) holds with $F$ the density function of the standard normal distribution $\mathcal{N}(0,1)$. This shows that the conditions in Lemma \ref{real graph} are satisfied with $E(N)=h(N)=(\log N)^{-1/2}$ which finishes the proof. \end{proof}
    
In particular, this yields a proof of Theorem \ref{thm:CLTintro} in the case of Hecke--Maa{\ss} forms for the modular group (with effective error term). Using the same argument as above, we obtain a proof of Theorem \ref{thm:CLTintro} in the case of holomorphic cusp forms for a general Fuchsian group $\Gamma$. This time we apply Proposition \ref{prop:main}, Theorem \ref{thm:normalnoneffective}, and Lemma \ref{complex graph} to the graph $G_N$ and 
    $$f_N(x) = \frac{L_f(x)}{c_f \sqrt{\log N}}, \quad g_N(y) = (-1)^{k/2+1} \frac{\mathcal{P}_f(y)}{c_f \sqrt{\log N}}. $$
    Furthermore, the constant $c_f>0$ is as in (\ref{eq:variance}).



\section{Non-vanishing of central values of $L$-functions}
\label{sec:application}

Denote by $\Df^+$ the set of positive fundamental discriminants, that is
\begin{align*}
    \Df^+ := \Bigg \{ D>0: \begin{array}{l}  D \equiv 1 \mods{4}\ \text{ and } D \text{ square-free or} \\ 
     D=4m, \text{ where } m \equiv 2,3 \mods{4} \text{ and } m \text{ square-free} \end{array} \Bigg \}.
\end{align*}

Let $D \in \Df^+$ and $K=\Q(\sqrt{D})$ the associated real quadratic field. Let $N$ be a square-free integer coprime with $D$ such that each prime divisor $p$ of $N$ splits in $K$. Hence there exists an integer $\alpha$ such that $D \equiv \alpha^2 \mods{4N}$. A quadratic form $Q=[a,b,c]$ is said to be {\it Heegner of level $N$} if $N\mid a$ and $b \equiv \alpha \mods{2N}$. Denote by $\mathcal{Q}_{N,D}$ the set of primitive quadratic forms Heegner of level $N$ and discriminant $D$. Then $\mathcal{Q}_{N,D}$ is stable under the action of $\Gamma_0(N)$ and we have isomorphisms
\begin{align*}
    \Gamma_0(N) \backslash \mathcal{Q}_{N,D} \xlongrightarrow{\sim} \mathrm{SL}_2(\Z) \backslash \mathcal{Q}_{1,D}
\end{align*}
and $$\Cl_D^+ \xlongrightarrow{\sim} \Gamma_0(N) \backslash \mathcal{Q}_{N,D},$$
where $\Cl_D^+$ is the narrow class group of $\Q(\sqrt{D})$. Therefore, for each $A \in \Cl_D^+$, there is an associated primitive geodesic $\mathcal{C}_A$ in $\mathcal{X}_0(N):=\Gamma_0(N) \backslash \H$ of length $2\log \epsilon_D$, we refer to \cite{darmon94} and \cite[Chapter 6]{popa06} for  details. We denote by $h^+(D)=|\Cl^+_D|$ the (narrow) class number and by $J=(\sqrt{D})\in \Cl_D^+$ the different.

For a fixed square-free $N$, we are interested in working with fundamental discriminants such that the above holds. A prime $p$ splits in $\Q(\sqrt{D})$ exactly if $(p,D)=1$ and $D$ is quadratic residue modulo $p$. Therefore, from the Chinese Remainder Theorem, there exists $I \subset \{1,2, \ldots, 4N-1 \}$ such that all prime divisors $p\mid N$ split in $\Q(\sqrt{D})$ exactly if $ (D  \ \mathrm{mod } \ 4N) \in I$.
We denote 
$$\DNf := \{ D\in \Df^+ : (D  \ \mathrm{mod } \ 4N) \in I\}$$
and 
$$\DNf(X) : = \{D \in \DNf : \ed \leq X \}.$$
From work of Hashimoto \cite{H13}, we know there exists a constant $a(N)$ such that
\begin{equation}\label{eq:Hashimoto}
    \sum_{D \in \DNf (X)} h^+(D) = a(N) \Li (X^2) + O_{N, \epsilon} \lr{X^{25/13 + \epsilon}}. 
\end{equation}

Let $f$ be a Hecke--Maa{\ss} newform of weight $k\in 2\Z_{\geq 0}$ for $\Gamma_0(N)$. For $D \in \DNf$, $K = \Q(\sqrt{D})$ and $\chi\in \widehat{\Cl_D^+}$ a class group character, we have Waldspurger formula in its explicit form (see \cite[Theorem 6.3.1]{popa06} and \cite[Lemma 6.14]{HumNor20}):
\begin{equation}\left| \sum_{A\in \Cl_D^+}\chi(A) \mathcal{P}_f(\CC_A)\right|^2= D^{1/2} c_{f,\chi(J)}L(f\otimes\theta_\chi,1/2)
,\end{equation}
where $\theta_{\chi}$ is the theta-series associated to $\chi$, and $c_{f,\chi(J)}\geq 0$ is a constant depending only on $f$ and $\chi(J)\in \{\pm1 \}$ such that $c_{f,\chi(J)}=0$ iff $k=0$ and $\epsilon(f)=-\chi(J)$. Note here that the geodesic periods used in \cite[(6.1.2)]{popa06} do indeed match the geodesic periods $\mathcal{P}_f(\CC_A)$ up to a constant depending only on $f$ as follows from the explicit formulas in \cite[Sec. 6.1]{popa06}. By applying Plancherel (orthogonality of characters) to the Waldspurger formula as above, we obtain
\begin{equation}
    \label{plancherel}
    \frac{D^{1/2}}{h^+(D)}\sum_{\chi \in \widehat{\Cl_D^+}} c_{f,\chi(J)}L(f\otimes\theta_\chi,1/2)=\sum_{A\in \Cl_D^+} \left| \mathcal{P}_f(\CC_A) \right|^2.
\end{equation}
Using this together with our previous results, we obtain the following non-vanishing result for central values of Rankin-Selberg $L$-functions.
\begin{theorem} There exists a constant $c=c(N)>0$ depending only on $N$ such that, as $X \to \infty$,
    \begin{equation}\label{eq:statecorollary}
        \frac{ |\{ D \in \DNf (X) : \exists \chi\in \widehat{\Cl_D^+}\text{ s.t. } L(1/2,f\otimes \theta_\chi)\neq 0\}|}{| \DNf (X)|}\geq c+o(1).\end{equation} 
\end{theorem}

\begin{proof}
    Let $\Df^+ (X):= \{ D \in \Df^+ : \ed \leq X \}$. We use the moments of $h(D)$ restricted to fundamental discriminants, as computed by Raulf \cite{raulf16}.  More precisely, we have that for each $k \in \N \cup \{ 0 \}$, there exists constants $C(k),\delta_k>0$ such that
    $$\sum_{D \in \Df^+(X)} h(D)^k = C(k) \int_2^X \lr{\frac{t}{\log t}}^k dt + O \lr{X^{k+1-\delta_k}}=\frac{C(k)}{k+1}\frac{X^{k+1}}{(\log X)^k}(1+o(1)).$$
 We denote 
    $$\mathcal{D}_{\mathrm{good}}(X) : = \left\{ D \in \DNf (X) : \exists \chi \in \widehat{\Cl_D^+} \text{ with }L(1/2, f \otimes \theta_{\chi} ) \neq 0 \right\} \subset \DNf (X). $$ 
    Note that if $D \in \DNf (X) \setminus \mathcal{D}_{\mathrm{good}}(X) $, we see from \eqref{plancherel} that $\mathcal{P}_f(\CC_A) = 0$ for all $A \in \Cl_D^+$. In particular,  
    \begin{equation*}
        \bigcup_{D \in \DNf (X) \setminus \mathcal{D}_{\mathrm{good}}(X)}  \{\mathcal{C}_A:A \in \Cl_D^+\} \subseteq \left \{ \mathcal{C} \subset \mathcal{X}_0(N) : \ell(\mathcal{C}) \leq 2 \log X, \mathcal{P}_f(\CC) =0 \right \}.
    \end{equation*}
    Fix $\epsilon>0$. Using Theorem \ref{thm:main} and the prime geodesic theorem (\ref{eq:sizeXY}), we obtain that for $X$ sufficiently large
    $$\sum_{D \in \DNf (X) \setminus \mathcal{D}_{\mathrm{good}}(X)} h^+(D) \leq \epsilon \frac{X^2}{\log X}.$$
    Using Cauchy--Schwarz, and the fact that $\mathcal{D}_{\mathrm{good}}(X) \subseteq \DNf(X) \subseteq \Df(X)$, it follows that for $X$ sufficiently large
    \begin{align*}
        (C(2)/3 + \epsilon) \frac{X^3}{(\log X)^2}  &\geq \sum_{D \in \Df^+ (X)} h^+(D)^2 \geq \sum_{D \in \mathcal{D}_{\mathrm{good}}(X)} h^+(D)^2\\
        &\geq \frac{ \lr{\sum_{D \in \mathcal{D}_{\mathrm{good}}(X)} h^+(D) }^2 }{|\mathcal{D}_{\mathrm{good}}(X)|} \geq \frac{(a(N)-\epsilon)^2 \frac{X^4}{(\log X)^2}}{|\mathcal{D}_{\mathrm{good}}(X)|}, 
    \end{align*}
with $a(N)$ as in eq.\ (\ref{eq:Hashimoto}).    This implies that for any $\epsilon>0$, we have for $X$ sufficiently large that  
    $$|\mathcal{D}_{\mathrm{good}}(X)|\geq \frac{(a(N)-\epsilon)^2}{C(2)/3+\epsilon} X.$$
   This implies that the lower bound (\ref{eq:statecorollary}) holds with  $c=\frac{a(N)^2}{C(2)/3} \beta$ where
    \begin{equation*}
        \beta := \liminf_{X\to \infty} \frac{X}{| \DNf (X)|}\geq\liminf_{X\to \infty} \frac{X}{| \Df^+ (X)|}\geq C(0)^{-1}  >0,
    \end{equation*}
    where $C(0)>0$ is as in Raulf's work \cite{raulf16} mentioned above.\end{proof}

\begin{remark}
    If one applies H\"{o}lder inequality instead of Cauchy--Schwarz one obtains $c =\lr{\frac{a(N)^k}{C(k)/(k+1)}}^{1/(k-1)}\cdot\beta$ for any $k\geq 1$. However, none of the constants $C(k)$ have been computed besides $k=0,1$, and the upper bounds of $C(k)$ from \cite{raulf16} would indicate fast decay as $k \to \infty$.
\end{remark}
\AtNextBibliography{\small}
\printbibliography[heading=bibintoc,title={References}]
\end{document}